\documentclass[12pt]{amsart}
\usepackage{pstricks, pst-node, pst-tree,pstricks-add,pst-eucl}
\usepackage{multido}
\usepackage{pst-text}
\usepackage{pst-bar}
\usepackage{amssymb}
\usepackage{pst-plot}
\usepackage{comment}
\usepackage{graphicx} 
\usepackage{hyperref}

\usepackage{amsmath}
\usepackage{latexsym}
\usepackage{ifthen}

\usepackage{yfonts}
\usepackage{tikz} 
\usetikzlibrary{shapes}
\usetikzlibrary{arrows}

\usepackage[inline]{enumitem}
\usepackage{amsaddr}

\theoremstyle{plain}
\newtheorem{prop}{Proposition}[section]
\newtheorem{thm}[prop]{Theorem}
\newtheorem{lem}[prop]{Lemma}

\theoremstyle{remark}

\theoremstyle{definition}
\newtheorem{defn}[prop]{Definition}
\newtheorem{exm}[prop]{Example}

\newrgbcolor{darkgreen}{0.0 0.39215686 0}
\newrgbcolor{LemonChiffon}{1.0 0.98 0.8}
\newrgbcolor{magenta}{1.0 0 1.0}
\newrgbcolor{mistyrose}{1.0 0.89 0.88}
\newrgbcolor{deeppink}{1.0 0.08 0.58}
\newrgbcolor{lightsalmon}{1.0 0.63 0.48}
\newrgbcolor{lightred}{0.7 0.0 0.0}
\newrgbcolor{lightblue}{0.68 0.85 0.90}
\newrgbcolor{indianred}{0.80  0.36  0.36}
\newrgbcolor{lightgreen}{0.56 0.93 0.56}
\newrgbcolor{darkred}{0.89928057554   0   0}

\newcommand{\formal}[1]{\ensuremath{\textsf{#1}}}

\newcommand{\dword}[1]{\textbf{#1}}

\newcommand{\product}{\ensuremath{\mu}}
\newcommand{\seq}[1]{\ensuremath{ \langle#1 \rangle}}
\newcommand{\sym}{\ensuremath{ \mathbf{S} }}

\newcommand{\VT}{\ensuremath{\mathcal{T}}}

\newcommand{\per}[1]{\ensuremath{ \mathsf{#1} }} 
\newcommand{\ps}[1]{\ensuremath{ \mathbf{#1} }}  

\newcommand{\wk}{\ensuremath{\formal{T}}}

\newcommand{\graph}[1]{\ensuremath{ \mathbf{#1} }}

\newcommand{\migt}{MIGT}

\newcommand{\walk}[1]{\seq{#1}}

\newcommand{\slide}{\ensuremath{\mathcal{S}}}
\newcommand{\dual}{\ensuremath{\mathcal{D}}}

\newcommand{\mindbod}[2]{\ensuremath{{ {#1}\choose{#2} }}}

\newcommand{\mact}{\ensuremath{ \ \textcircled{m} \ }}
\newcommand{\bact}{\ensuremath{ \ \textcircled{b} \ }}

\newcommand{\id}{\ensuremath{\mathcal{I}}}

\newcommand{\trails}[1]{\ensuremath{\mathcal{#1}}}

\newcommand{\arc}[1]{\ensuremath{\overrightarrow{(#1)}}}

\newcommand{\edge}[1]{\ensuremath{\overline{#1}}}

\newcommand{\bij}{\ensuremath{\mathcal{B}}}

\newcommand{\cpart}[2]{\ensuremath{#1 \bigcirc #2}}
\newcommand{\fpart}[2]{\ensuremath{#1 \bigtriangleup #2}}

\newcommand{\cnorm}{\formal{C-Index}}
\newcommand{\fnorm}{\formal{T-Index}}

\newcommand{\rd}{-}

\newcommand{\F}{\ensuremath{\mathcal{F}}}
\newcommand{\Fdown}[1]{\ensuremath{\F^{(#1 \ldots 1)}}}
\newcommand{\Fup}[1]{\ensuremath{\F^{(1 \ldots #1)}}}

\newcommand{\junk}[1]{}

\title{Three Graph Duals and A Bijection}

\author{Nikolaos Apostolakis}
\address{Department of Mathematics and Computer Science, Bronx Community College (CUNY), Bronx, NY, U.S.A.}
\email{nikolaos.apostolakis@bcc.cuny.edu}

\author{Kerry Ojakian}
\address{Department of Mathematics and Computer Science, Bronx Community College (CUNY), Bronx, NY, U.S.A.}
\email{kerry.ojakian@bcc.cuny.edu}

\date{March 19, 2017}

\begin{document}
\maketitle

\tikzstyle{vertex}=[circle,fill=black!25,minimum size=20pt,inner sep=0pt]
\tikzstyle{minivertex}=[circle,fill=black!25,minimum size=10pt,inner sep=0pt]
\tikzstyle{edge} = [draw,thick,-]
\tikzstyle{miniedge} = [draw,dashed,-]
\tikzstyle{weight} = [font=\small]

\begin{abstract}

We develop a notion of a dual of a graph, generalizing the definition of Goulden and Yong (which only applied to trees), and reproving their main result using our new notion.  We in fact give three definitions of the dual: a graph-theoretic one, an algebraic one, and a combinatorial ``mind-body'' dual, showing that they are in fact the same, and are also the same (on trees)
 as the topological dual developed by Goulden and Yong.  Goulden and Yong use their dual to define a bijection between the vertex labeled trees and the factorizations of the permutation $(n, \ldots, 1)$ into $n-1$ transpositions, showing that their bijection has a particular structural property.    We reprove their result using our dual instead.

\smallskip
\noindent \textsc{Keywords.} Multi-Graph, Trail, Transposition, Dual, Bijection

\end{abstract}

\section{Introduction}

The basic objects we investigate are factorizations of permutations into
transpositions.
By $\sym_n$ we mean the symmetric group on the set $[n]= \{ 1, 2, \ldots, n\}$;
for us, all multiplication is from left to right. To
refer to ``factorizations'' precisely we introduce the notion of 
a \emph{transposition sequence}.
\begin{defn}
A \dword{transposition sequence} (over $\sym_n$) is a sequence of transpositions
$\per{s} =\seq{\per{s_1}, \ldots, \per{s_m}}$.  We write $\product(\per{s})$,
called the \dword{product} of \per{s}, to mean the permutation resulting from the product: 
$\per{s_1} \cdot \per{s_2} \cdots \per{s_m}$.
\end{defn}

\begin{exm} \label{ex_trans_seq}
The sequence $\ps{s} = \seq{(3,4), (1,3), (1,2), (3,4), (2,3)}$ is a transposition sequence over $\sym_4$, and its product
$\product(\per{s}) = (4, 3, 2, 1)$.
\end{exm}
We use standard notation, letting $(n, \ldots, 2, 1)$ represent the permutation
mapping $n$ to $n-1$, $n-1$ to $n-2$, and so on, with $1$ mapped to $n$.
This permutation has a number of factorizations into $n-1$ transpositions.
For example, the permutation $(3, 2, 1)$ in $\sym_3$ has exactly three distinct factorizations into $2$ transpositions, represented by the following transposition
sequences:
$\seq{(1,2), (2, 3)}$, $\seq{(2,3), (1, 3)}$ and $\seq{(1, 3), (1,2)}$.  
D{\'e}nes~\cite{Denes1959} showed that in general there are exactly $n^{n-2}$ factorizations of $(n, \ldots, 2, 1)$ into $n-1$ transpositions.  Since it is well-known that there are also $n^{n-2}$ vertex labeled trees on $n$ vertices, D{\'e}nes suggested the project of finding interesting bijections between these factorizations and these trees.
While interesting in its on right, the project posed by D{\'e}nes is further motivated by fitting it into a broader context suggested by \cite{GouldenJackson1997} and \cite{GouldenYong}.  A factorization of $(n, \ldots, 2, 1)$ into $n-1$ transpositions is in fact what is called a \emph{minimal transitive factorization}; any permutation of $\sym_n$ can be factored into its minimal transitive factorizations.
Minimal transitive factorizations are of interest because of
their connection to topology (for example, see
 \cite{Arnold1996} and \cite{GJV2001}).

\begin{defn} \
\label{def_FT}

\begin{itemize}

\item
Let \Fdown{n} be the set of length $n-1$ transposition sequences over $\sym_n$, with product $(n, \ldots, 2, 1)$.

\item
Let 
$\VT_n$ be the set of trees on $n$ vertices, so that each vertex gets a distinct label from the set $[n]$.

\end{itemize}

\end{defn}

\noindent
Using this terminology, D{\'e}nes' challenge is to find bijections between \Fdown{n} and 
$\VT_n$. Moszkowski~\cite{Moszkowski1989} found a bijection in 1989; then in 1993
Goulden and Pepper~\cite{GouldenPepper1993} found a different bijection.
However, arguably the nicest bijection is developed in 2002 by Goulden and Yong \cite{GouldenYong}; in this bijection, various structural properties are preserved.  
Essential in the bijection of \cite{GouldenYong} is their definition of the dual of a tree, defined topologically.  
The main point of our work is an alternative definition of the dual and the bijection, along with an alternative proof that the bijection has
the desired structural properties.  Our bijection will turn out to be the same as theirs and
we will give three definitions of the dual which will 
coincide with one another and, on trees, with their definition.
We also give credit to Herando Mart\'{i}n \cite{Martin1999} who, in 1999, independently developed the dual from \cite{GouldenYong}, though the work of \cite{Martin1999} then goes in a different direction from that of \cite{GouldenYong}.

In Section~\ref{sec_mb} we interpret a transposition sequence as instructions for 
a sequence of mind-body swaps (currently science fiction), developing our first  definition of dual.
In Section~\ref{sec_graphs} we give a graph-theoretic interpretation of transposition
sequences, and in Section~\ref{sec:greedy} we give an equivalent second definition of the dual in
the graph-theoretic context.
In Section~\ref{sec_alg} we give our third definition of dual, an algebraic characterization,
which leads to a simple graph-theoretic algorithm for computing the dual.
In Section~\ref{sec_alt} we define a bijection between \Fdown{n} and $\VT_n$ that 
enjoys the same nice structural properties as the bijection from \cite{GouldenYong}.
In Section~\ref{sec_gydual} we prove that our dual (when restricted to trees) is in fact the same as the topological dual of \cite{GouldenYong} and \cite{Martin1999}.
Our dual is interesting in its own right, and because it applies to all finite graphs, coincides with \cite{GouldenYong} and \cite{Martin1999}  for trees, and allows us to give very different proofs for results that use the dual.

\section{Mind-Body Interpretation}
\label{sec_mb}

Following Evans and Huang~\cite{EvansHuang2014}, which is based on some science fiction shows, we can view a transposition sequence as instructions 
for a sequence of mind-body swaps; we will find this interpretation technically useful and interesting in its in own right.
We imagine that there is a mind-swapping machine (which we just call \dword{The Machine}), with positions for two people.  When we operate The Machine, we don't see anything happen, but the minds insides the two bodies are swapped.  In fact in this scenario, properly speaking, it is not clear where the person is, since their mind may not be in their body. Thus 
we should say that two bodies (say $B_1$ and $B_2$) enter The Machine (each body is currently associated to some mind, say mind $M_1$ is in $B_1$ and mind $M_2$ is in $B_2$); after the operation of The Machine, body $B_1$ contains mind $M_2$ and body $B_2$ contains mind $M_1$.
To keep track of the current state of affairs we use a \emph{Mind-Body Assignment}. 
\begin{defn}

A \dword{Mind-Body Assignment} (over $n$) is a permutation in $\sym_n$ written using inline notation, i.e. the permutation mapping $M_k$ to $B_k$ for $k = 1, \ldots, n$ is written as \mindbod{M_1, \ldots, M_n}{B_1, \ldots, B_n}.
The top elements are called the \emph{minds} and the bottom elements are called the \emph{bodies}. We say that $M_i$ is \dword{above} $B_i$ and 
$B_i$ is \dword{below} $M_i$.
Note that the order of the $M_i$ or $B_i$ is irrelevant; all that matters is the assignment.

\end{defn}

\begin{exm}
The Mind-Body Assignment \mindbod{1, 2, 3, 4}{4, 1, 2, 3}
indicates that mind 1 is in body 4, mind 2 is in body 1, mind 3 is in body 2, and mind 4 is in body 3.
As order does not matter, the mind-body assignment \mindbod{2, 3, 4, 1}{1, 2, 3, 4}
is the same as \mindbod{1, 2, 3, 4}{4, 1, 2, 3}.
\end{exm}

To make our discussion precise, we will in fact
view a transposition sequence as \emph{either} instructions for a series of mind swaps or as instructions for a series of body swaps.

\begin{defn}
We define two operations on the Mind-Body Assignments.
Let $\per{s} = (x,y)$ be a transposition in $\sym_n$ and let $A$ be a Mind-Body Assignment over $n$.

\begin{itemize}

\item The \dword{Mind-Swapping Operation}: We define $A \mact \per{s}$ to be the Mind-Body Assignment in which the order of the bodies is unchanged, and minds $x$ and $y$ are swapped.

\item
The \dword{Body-Swapping Operation}: We define $A \bact \per{s}$ to be the Mind-Body Assignment in which the order of the minds is unchanged, and bodies $x$ and $y$ are swapped.

\item If $\per{s} =\seq{\per{s_1}, \ldots, \per{s_m}}$ is a transposition sequence, we write $A \mact \ps{s}$ to mean
$(A \mact \per{s_1}) \mact{\per{s_2}} \ldots$, and write $A \bact \ps{s}$ for 
$(A \bact \per{s_1}) \bact{\per{s_2}} \ldots$

\end{itemize}

\end{defn}

\begin{exm} \ \label{ex_bact_mact}

\begin{itemize}

\item
\mindbod{1, 2, 3, 4}{1, 2, 3, 4} \bact \seq{(3,4), (1,3)} $=$ \mindbod{1,2,3,4}{3,2,4,1}

\item
\mindbod{1, 2, 3, 4}{1, 2, 3, 4} \mact \seq{(3,4), (1,3)} $=$
\mindbod{3,2,4,1}{1,2,3,4}

\end{itemize}

\end{exm}

Note that The Machine, as discussed above, is formalized by the body-swapping operation.  
To see this, consider Example~\ref{ex_bact_mact} and what happens if The Machine follows the instructions \seq{(3,4), (1,3)} .  First bodies $3$ and $4$ step
into The Machine, and then bodies $1$ and $3$ step into The Machine.
The first operation makes it so that mind $4$ is now in body $3$ and mind $3$ is now in
body $4$.  For the second operation, bodies $1$ and $3$ step into the machine,
resulting in body $3$ having mind $1$ and body $1$ having mind $4$ (since mind $4$ was 
 in body $3$ at that point).  Note that in the example, the body-swapping operation 
accomplishes exactly this.

It is natural to assume that the original position of the minds is such that mind $k$ is in body $k$.  This original position is represented by the identity permutation written as a Mind-Body Assignment:
We let $\id_n$ be the identity Mind-Body Assignment \mindbod{1, 2, \ldots, n}{1, 2, \ldots, n}; if $n$ is clear from context, we may just write \id.
Since we will generally view the instructions as a sequence of body swaps, it will be interesting to record the \emph{effects} of such instructions as a sequence of Mind-Body Assignments.

\pagebreak

\begin{defn}
Given a transposition sequence $\ps{s} = \seq{\per{s_1}, \ldots, \per{s_m}}$, its
\dword{Corresponding Mind-Body Sequence} is the sequence of Mind-Body Assignments
\seq{A_0, A_1, \ldots, A_m}, where 
\begin{itemize}

\item
$A_0 = \id$, and

\item
$A_k = \id \bact \seq{\per{s_1}, \ldots, \per{s_k}}$.

\end{itemize}


\end{defn}
 
\begin{exm} \label{ex_mindbod_seq}
The Corresponding Mind-Body Sequence of the transposition sequence in 
Example~\ref{ex_trans_seq} is \seq{A_0, A_1, A_2, A_3, A_4, A_5}, where

\setlength\tabcolsep{10pt}
\begin{tabular}{lll}
$A_0 = \mindbod{1, 2, 3, 4}{1, 2, 3, 4}$ & 
$A_1 = \mindbod{1, 2, 3, 4}{1, 2, 4, 3}$ & 
$A_2 = \mindbod{1, 2, 3, 4}{3, 2, 4, 1}$ \\
$A_3 = \mindbod{1, 2, 3, 4}{3, 1, 4, 2}$ &
$A_4 = \mindbod{1, 2, 3, 4}{4, 1, 3, 2}$ &
$A_5 = \mindbod{1, 2, 3, 4}{4, 1, 2, 3}$ \\
\end{tabular}

\end{exm}

\noindent
Immediately from the definitions, we have the following. 

\begin{lem}
\label{lem_acts_are_products}

Let $\per{s} = (x,y)$ be a transposition and let $A$ be a Mind-Body Assignment.  Then
$A \mact (x,y) = (x,y) \cdot A$ \ and \ 
$A \bact (x,y) = A \cdot (x,y)$.

\end{lem}

\noindent
Also, note that if we start with the identity Mind-Body Assignment \id, then applying the Mind-Swapping Operation with transposition sequence \ps{s}, changes the top in some way, while the Body-Swapping Operation using \ps{s} does exactly the same thing to the bottom, so we have the following.

\begin{lem}
\label{lem_mact_inverse_bact}
If \ps{s} is a transposition sequence, then
$\id \mact \ps{s} = (\id \bact \ps{s})^{-1}$.
\end{lem}

\noindent
Example~\ref{ex_bact_mact} exhibits the last lemma.
The last lemma is also an immediate consequence of repeated application of Lemma~\ref{lem_acts_are_products}.

We now define a notion of dual, which converts a sequence of body switches into a corresponding sequence of mind switches.

\newcommand{\mb}{\formal{MB}}

\begin{defn}
\label{def_mind_body_duals} \

\begin{itemize}

\item
Given a Mind-Body Assignment $A$, and a transposition $\per{s} = (x,y)$, we define the \dword{Mind-Body Dual} of \per{s} in $A$, denoted $\mb_A (\per{s})$ to be the transposition $(x',y')$, where in $A$, $x'$ is above $x$ and $y'$ is above $y$.

\item
Suppose $\ps{s} =\seq{\per{s_1}, \ldots, \per{s_m}}$ is a transposition sequence
and \seq{A_0, A_1, \ldots, A_m} is its Corresponding Mind-Body Sequence.
The \dword{Mind-Body Dual} of \ps{s}  is the transposition sequence $\ps{s'} =\seq{\per{s'_1}, \ldots, \per{s'_m}}$, where 
$\per{s'_k} = \mb_{A_{k-1}}(\per{s_k})$.
We say that the Mind-Body Dual of \per{s_k} in \ps{s} is \per{s'_k}.

\end{itemize}

\end{defn}

\begin{exm} \label{ex_dual}
To calculate the Mind-Body Dual \seq{\per{s'_1}, \ldots, \per{s'_5}} of \\
\seq{(3,4), (1,3), (1,2), (3,4), (2,3)}, 
the transposition sequence of Example~\ref{ex_trans_seq}, we use its
Corresponding Mind-Body Sequence from Example~\ref{ex_mindbod_seq}.

\begin{itemize}

\item Above $(3,4)$ in $A_0$ is $(3,4)$ so $\per{s'_1} = (3,4)$.

\item Above $(1,3)$ in $A_1$ is $(1,4)$ so $\per{s'_2} = (1,4)$.

\item Above $(1,2)$ in $A_2$ is $(2,4)$ so $\per{s'_3} = (2,4)$.

\item Above $(3,4)$ in $A_3$ is $(1,3)$ so $\per{s'_4} = (1,3)$.

\item Above $(2,3)$ in $A_4$ is $(3,4)$ so $\per{s'_5} = (3,4)$.

\end{itemize}
So the Mind-Body Dual is \seq{(3,4), (1,4), (2,4), (1,3), (3,4)}.

\end{exm}

If we view a transposition sequence \ps{s} as a sequence of body-swapping instructions, then \ps{s} describes the bodies we would see entering The Machine.  From the definition of Mind-Body Dual, we can see that \ps{s'} describes the sequence of corresponding mind switches.  We make this idea precise in the next lemma, then after its proof, we give a more detailed interpretation.
\begin{lem}
\label{lem_actions_and_dual}

For any transposition sequence \ps{s} and Mind-Body Assignment $A$, we have:
$A \bact \ps{s} =  A \mact \ps{s'}$ \ and \
$A \mact \ps{s} =  A \bact \ps{s'}$.

\end{lem}
\begin{proof}

We only prove the first equality, as the proof of the 
second is completely analogous.
The proof is essentially repeated application of the following \emph{observation}:
\begin{quote}
For any Mind-Body Assignment $X$ and transposition \per{t}, we have that
$X \bact \per{t} = X \mact \mb_{X}(\per{t})$.

\end{quote}

\noindent
Let $\per{s} =\seq{\per{s_1}, \ldots, \per{s_m}}$.  We proceed by induction, showing that for $k$ up to $m$, 
$A \bact \seq{\per{s_1}, \ldots, \per{s_k}} =  A \mact \seq{\per{s'_1}, \ldots, \per{s'_k}}$. Consider the inductive step,
where $B = A \bact \seq{\per{s_1}, \ldots, \per{s_{k-1}}}$.

\begin{alignat*}{2}
A \bact \seq{\per{s_1}, \ldots, \per{s_k}} 
    & = A \bact \seq{\per{s_1}, \ldots, \per{s_{k-1}}} \bact \per{s_k} && \text{} \\
    & =  A \bact \seq{\per{s_1}, \ldots, \per{s_{k-1}}} \mact \mb_{B}(\per{s_k}) && \text{, by above observation} \\		
    & = A \mact \seq{\per{s'_1}, \ldots, \per{s'_{k-1}}} \mact \mb_{B}(\per{s_k}) && \text{, by inductive hypothesis} \\ 
		& = A \mact \seq{\per{s'_1}, \ldots, \per{s'_k}} && \text{ }
\end{alignat*}
\end{proof}

We can now give a nice interpretation of the Mind-Body Dual in terms of The Machine.  The Mind-Body Assignment \id \ represents the original situation where everyone has their own mind.  A transposition sequence 
$\ps{s} = \seq{\per{s_1}, \ldots, \per{s_m}}$ can be viewed as a sequence of instructions for who (i.e. which bodies) go into The Machine.  By definition, the Mind-Body Dual $\ps{s'} = \seq{\per{s'_1}, \ldots, \per{s'_m}}$ describes exactly the same sequence of instructions but indicating which minds step into The Machine at each step.  In particular, if $\per{s_k} = (x, y)$ and $\per{s_k'} = (x',y')$, that means that the $k^{th}$ swap is between $x$ and $y$ when viewed from the bodies' point of view, and between $x'$ and $y'$ when viewed from the minds' point of view.
In imagining the operation of The Machine, the sequence of instructions of \ps{s} is clearly \emph{visible} by watching which bodies step into The Machine.  We do not (unless we had a special mind-detecting device) see which minds step into The Machine, so the Mind-Body Dual \ps{s'} reveals the process from the point of view of the \emph{non-visible} minds.  Lemma~\ref{lem_actions_and_dual} proves that the invisible sequence of 
mind swaps given by the instructions of \ps{s'} accomplishes exactly the same result
as the visible sequence of body swaps given by the instructions of \ps{s}.

As we would hope, the next lemma states that the notion of duality is idempotent
(an alternative algebraic proof of this lemma is given at the end of section~\ref{sec_alg}).
\begin{lem} \label{lemDualTwice}
For any transposition sequence \per{s} we have $\per{s''} = \per{s}$.
\end{lem}

\begin{proof}

We proceed by induction on the length of the transposition sequence.  For the inductive step, suppose we want to show the claim for  
$\ps{s} = \seq{\per{s_1}, \ldots, \per{s_m}}$, inductively assuming that
$\seq{\per{s_1}, \ldots, \per{s_{m-1}}} =
\seq{\per{s''_1}, \ldots, \per{s''_{m-1}}}$.
Let $A = \id \bact \seq{\per{s_1}, \ldots, \per{s_{m-1}}}
= \id \bact \seq{\per{s''_1}, \ldots, \per{s''_{m-1}}}$.
By two applications of Lemma~\ref{lem_actions_and_dual},
$A \bact \per{s_m} = A \bact \per{s''_m}$.
Since \per{s_m} and \per{s''_m} are just transpositions,
having $A \bact \per{s_m} = A \bact \per{s''_m}$, implies 
$\per{s_m} = \per{s''_m}$.
Thus 
$\seq{\per{s_1}, \ldots, \per{s_{m-1}, \per{s_m}}} =
\seq{\per{s''_1}, \ldots, \per{s''_{m-1}}, \per{s''_m}}$.
\end{proof}

Note that the product of the permutation sequence from Example~\ref{ex_trans_seq} is
$(4,3,2,1)$, while the product of its Mind-Body Dual, from Example~\ref{ex_dual},
is the inverse $(1,2,3,4)$.  The next lemma points out that this is always the case.
\begin{lem} \label{lem_inverse} For any transposition sequence \ps{s},
\product(\ps{s}) and \product(\ps{s'}) are inverses. 
\end{lem}
\begin{proof}

\begin{alignat*}{2}
\product(\ps{s}) & = \id \bact \ps{s} && \text{, by Lemma~\ref{lem_acts_are_products}} \\ 
   & =  \id \mact \ps{s'} && \text{, by Lemma~\ref{lem_actions_and_dual}} \\ 
   & =  (\id \bact \ps{s'})^{-1} && \text{, by Lemma~\ref{lem_mact_inverse_bact}} \\
	 & = \product(\ps{s'})^{-1} && \text{, by Lemma~\ref{lem_acts_are_products}} 
\end{alignat*}
\end{proof}

\section{Graphs and Trails}
\label{sec_graphs}

We now interpret transposition sequences as labeled graphs, focusing on particularly significant \emph{trails} on these graphs.
For us, a \dword{graph}  will always mean a finite, loop-less graph, with multi-edges allowed.
We will occasionally refer to \emph{unlabeled} graphs, however we will usually work with
\dword{labeled graphs}, by which we always mean a graph whose $n$ vertices are labeled by $[n]$ (each vertex receiving a distinct label),
and whose $m$ edges are labeled by $[m]$ (each edge receiving a distinct label).
As is commonly done (for example, in \cite{GouldenYong}) we view a transposition sequence (over $\sym_n$)  $\per{s} = \seq{\per{s_1}, \ldots, \per{s_m}}$
as a labeled graph with vertex set $[n]$, with $m$ edges: For each transposition $\per{s_i} = (x_i, y_i)$ we create a corresponding edge between $x_i$ and $y_i$, labeled by $i$.  By the \dword{product} of a labeled graph, we mean the product of its 
associated transposition sequence.
Figure~\ref{graphOfTranSeq} displays the transposition sequence of Example~\ref{ex_trans_seq} as a graph.

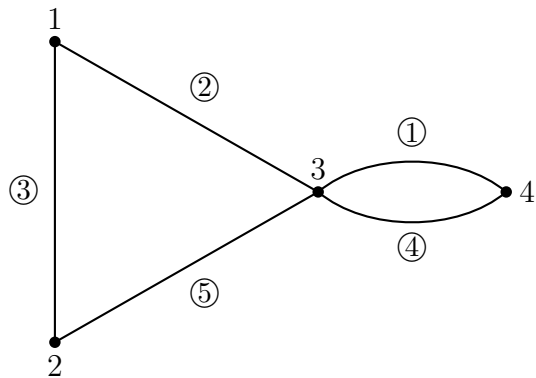
\begin{figure}
\begin{center}
\begin{pspicture}(-.3,-.5)(6.5,4.5)
\rput(0,4){\rnode{n1}{\psdot[dotscale=1.2](0,0)}}  
\uput[90](0,4){$1$}
\rput(0,0){\rnode{n2}{\psdot[dotscale=1.2](0,0)}}  
\uput[-90](0,0){$2$}
\rput(3.5,2){\rnode{n3}{\psdot[dotscale=1.2](0,0)}}  
\uput[90](3.5,2){$3$}
\rput(6,2){\rnode{n4}{\psdot[dotscale=1.2](0,0)}}  
\uput[0](6,2){$4$}
\ncline{n1}{n2}
\nbput{\raisebox{.5pt}{\textcircled{\raisebox{-.9pt} {\small $3$}}}}
\ncline{n1}{n3}
\naput{\raisebox{.5pt}{\textcircled{\raisebox{-.9pt} {\small $2$}}}}
\ncline{n2}{n3}
\nbput{\raisebox{.5pt}{\textcircled{\raisebox{-.9pt} {\small $5$}}}}
\ncarc[arcangle=40]{n3}{n4}
\naput{\raisebox{.5pt}{\textcircled{\raisebox{-.9pt} {\small $1$}}}}
\ncarc[arcangle=40]{n4}{n3}
\naput{\raisebox{.5pt}{\textcircled{\raisebox{-.9pt} {\small $4$}}}}
\end{pspicture}
\end{center}
\caption{The Transposition Sequence From Example~\ref{ex_trans_seq} Written As A Labeled Graph.}
\label{graphOfTranSeq}
\end{figure}

Following usual definitions, we define a \dword{trail} to be any meandering through the graph, with no restrictions except that an edge cannot be repeated; i.e. a
trail is a non-empty sequence \walk{v_1, e_1, v_2, e_2, \dots, v_{k-1}, e_{k-1}, v_k} such that each $v_i$ is a vertex, each $e_i = \{v_i, v_{i+1}\}$ is an edge, and no edge is repeated.
A \dword{trivial trail} is a trail which just consists of a vertex, and no edges.
Notice that trails are ordered with a \emph{start vertex} of $v_1$ and \emph{end vertex} of $v_k$.

We begin with the trails in labeled graphs which are of fundamental interest, namely those
whose edge labels greedily increase as little as possible.

\begin{defn}
Consider a labeled graph \graph{G} and any vertex $u$.
The 
\dword{Minimal Increasing Greedy Trail} (\migt) starting at $u$ is the trail starting at $u$,
always taking the smallest edge that is larger than any previous edges used.
The trail ends on the vertex from which it cannot move on.
This (unique) trail is referred to by $\wk^{(\graph{G})}_u$, where we just write
$\wk_u$ if \graph{G} is apparent from context.

\end{defn}

\begin{exm} \label{ex_migt}
Referring to the graph in Figure~\ref{graphOfTranSeq}, $\wk_3$ is the following trail:
$
\seq{3, \edge{1}, 4, \edge{4}, 3, \edge{5}, 2}
$, where we write \edge{e} to refer to the unique edge with label $e$.
\end{exm}

The fact that we have a trail from $x$ to $y$ tells us that the permutation takes $x$ to $y$ and furthermore gives us the ``trajectory'' taken by the element $x$ to arrive at $y$.
We make this simple, but interesting point precise.
It will be useful to refer to the \dword{contraction} of a sequence
\seq{a_1, \ldots, a_m} as the sequence arrived at by replacing any maximal subsequence
of consecutive entries $x_i, x_{i+1}, \ldots, x_j$ such that 
$x_i = x_{i+1} = \cdots = x_j$ by $x_i$; for example, the contraction of 
\seq{3,4,4,4,3,2,2} is \seq{3,4,3,2}.

\begin{defn}
Suppose $\ps{s} = \seq{\per{s_1}, \ldots, \per{s_m}}$ is a transposition sequence over 
$\sym_n$, and $x \in [n]$. 
The \dword{trajectory} of $x$ in \ps{s} 
is the contraction of \seq{x_0,x_1, \ldots, x_m}, where
$x_0 = x$ and for $k = 1, \ldots, m$, 
$\product_k = \product(\seq{\per{s_1}, \ldots, \per{s_k}})$, and 
$x_k$ is the result of applying the permutation $\product_k$ to $x$.  

\end{defn}

\begin{exm} \label{ex_traj}
Recall the transposition sequence from Example~\ref{ex_trans_seq},
$\ps{s} = \seq{(3,4), (1,3), (1,2), (3,4), (2,3)}$.
The trajectory of $3$ in \ps{s} is \seq{3, 4, 3, 2}.

\end{exm}
\noindent
We discuss two ways to understand trajectories: via trails in graphs and via the mind-body interpretation.

To view trajectories as trails, consider a transposition sequence 
$\ps{s} = \seq{\per{s_1}, \ldots, \per{s_m}}$ and some $x \in [n]$.
Consider the unique maximum length subsequence \seq{\per{s_{i_1}}, \ldots, \per{s_{i_k}}} such
that $\per{s_{i_1}} = (x, x_1), \per{s_{i_2}} = (x_1, x_2), \ldots, \per{s_{i_k}} = (x_{k-1}, x_k)$.
The trajectory of $x$ must then be \seq{x, x_1, \ldots, x_k}. In reference to that
last example, note that the subsequence corresponding to $3$ is 
\seq{(3,4), (4,3), (3,2)}, which means that $\wk_3$ will start at $3$, then 
go to $4$, then go back to $3$, and finally go to $2$. 
This discussion illustrates the following general fact.

\begin{lem} \label{lem_trail_trajectory}
Given any transposition sequence \ps{s} over $\sym_n$ and any $x \in [n]$,
the sequence of vertices in $\wk_x$ (in that order) is exactly the same as the  trajectory of $x$ in \ps{s}.
In particular, \product(\ps{s}) maps $x$ to $y$ if and only if the \migt \ starting at $x$ ends at $y$.

\end{lem}

\begin{exm}
The vertices of $\wk_3$ from Example~\ref{ex_migt} form the sequence \seq{3,4,3,2},
exactly the trajectory of 3 from Example~\ref{ex_traj}.
Also $\product(\ps{s}) = (4, 3, 2, 1)$, so 3 is mapped to 2, which we also see in $\wk_3$, which starts at 3 and ends at 2.
\end{exm}

For a second understanding of trajectories,
we can interpret the trajectory of an element as the sequence of bodies through which a mind passes.  To make this explicit, we state the following lemma, essentially the
same point as Lemma~\ref{lem_trail_trajectory}.

\begin{lem}
Consider a transposition sequence over 
$\sym_n$, and its Corresponding Mind-Body Sequence \seq{A_0, A_1, \ldots, A_m}; let $x \in [n]$.  Let $b_k$ be  the body that $x$ is above in $A_k$.
The contraction of \seq{b_0, b_1, \ldots, b_m} is the trajectory of $x$.
\end{lem}

\begin{exm}
Consider \ps{s} from Example~\ref{ex_trans_seq} and consider mind $3$. By looking at the Corresponding Mind-Body Sequence for \ps{s} from Example~\ref{ex_mindbod_seq}, we see that mind $3$ starts in body $3$, then moves to body $4$, then back to body $3$, then to body $2$.  The sequence of bodies through which mind $3$ passes is \seq{3, 4,3,2}, exactly the trajectory of $3$ in \ps{s}, as noted in Example~\ref{ex_traj}.

\end{exm}

We now consider a generalization of the \migt{s}, in order to better understand their structure.  While slightly tangential to the main thread of this paper, we believe this short diversion is interesting, and in fact the first author has 
other work building on it.

\begin{defn}
A \dword{Trail Double Cover} of a graph is a set of trails such that:

\begin{itemize}

\item
A unique trail begins at each vertex, and

\item
Every edge of the graph is used by exactly two trails.

\end{itemize}

\end{defn}

\noindent
An example of a Trail Double Cover is the graph in Figure~\ref{migtGraph}.  
Figure~\ref{migtGraph} is in fact the graph from 
Figure~\ref{graphOfTranSeq} with its \migt{s} drawn in.

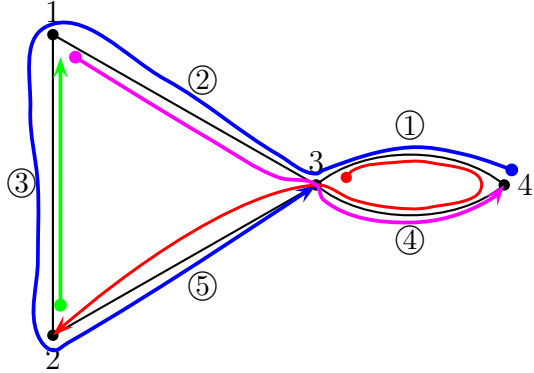
\begin{figure}
\begin{center}
\begin{pspicture}(-.3,-.5)(6.5,4.5)
\rput(0,4){\rnode{n1}{\psdot[dotscale=1.2](0,0)}}  
\uput[90](0,4){$1$}
\rput(0,0){\rnode{n2}{\psdot[dotscale=1.2](0,0)}}  
\uput[-90](0,0){$2$}
\rput(3.5,2){\rnode{n3}{\psdot[dotscale=1.2](0,0)}}  
\uput[90](3.5,2){$3$}
\rput(6,2){\rnode{n4}{\psdot[dotscale=1.2](0,0)}}  
\uput[0](6,2){$4$}
\ncline{n1}{n2}
\nbput{\raisebox{.5pt}{\textcircled{\raisebox{-.9pt} {\small $3$}}}}
\ncline{n1}{n3}
\naput{\raisebox{.5pt}{\textcircled{\raisebox{-.9pt} {\small $2$}}}}
\ncline{n2}{n3}
\nbput{\raisebox{.5pt}{\textcircled{\raisebox{-.9pt} {\small $5$}}}}
\ncarc[arcangle=40]{n3}{n4}
\naput{\raisebox{.5pt}{\textcircled{\raisebox{-.9pt} {\small $1$}}}}
\ncarc[arcangle=40]{n4}{n3}
\naput{\raisebox{.5pt}{\textcircled{\raisebox{-.9pt} {\small $4$}}}}
\pscurve[linecolor=blue,linewidth=.05,arrowsize=.2]{*->}(6.1,2.2)(5.5,2.4)(4.75,2.5)(3.6,2.2)(3.5,2.15)(3,2.4)(1.5,3.4)(-.1,4.1)(-.3,3)(-.2,2)(-.2,1)(-.2,0)(0,-.2)(.2,-.1)(3.5,2)
\pscurve[linecolor=red,linewidth=.04,arrowsize=.2]{*->}(3.9,2.1)(4,2.2)(4.55,2.3)(5,2.3)(5.7,2)(5,1.7)(4.5,1.7)(4,1.8)(3.5,2)(0,0)
\psline[linecolor=green,linewidth=.05,arrowsize=.2]{*->}(.1,.4)(.1,3.7)
\pscurve[linecolor=magenta,linewidth=.05,arrowsize=.2]{*->}(.3,3.7)(2.1,2.6)(3,2.1)(3.5,2)(3.6,1.8)(4.75,1.5)(5.8,1.8)(6,2)
\end{pspicture}
\end{center}
\caption{Graph From Figure~\ref{graphOfTranSeq} With Its \migt{s}.}
\label{migtGraph}
\end{figure}

\begin{lem} \label{lemmagreedypartition}
The set of \migt{s} of a labeled graph is a Trail Double Cover.

\end{lem}
\begin{proof}

For the first requirement on being a Trail Double Cover, we note that
there is only one \migt \ trail that starts at a vertex, since there is at most
one smallest edge at a vertex.

We proceed by induction on the number of edges to prove the second requirement:
that each edge is used by exactly two trails.  Let our graph be \graph{G}.
For the inductive step, remove the edge \edge{1}, i.e. the edge labeled by 1;
call the resulting graph $\graph{G}^{\rd}$.
Suppose edge \edge{1} in graph \graph{G} consists of vertices $x$ and $y$.
By inductive hypothesis,
in $\graph{G}^{\rd}$, every edge is used by exactly two trails; the fact that the edge
labels of $\graph{G}^{\rd}$ start at $2$ instead of $1$, has no effect on the argument.
Define trails $\wk_x^{\rd} = \wk_x^{(\graph{G}^{\rd})}$ and 
$\wk_y^{\rd} = \wk_y^{(\graph{G}^{\rd})}$.
Thus in \graph{G}, $\wk_x = \wk_x^{(\graph{G})}$ starts at $x$, follows edge
\edge{1} to $y$ and then does exactly that $\wk_y^{\rd}$ does in $\graph{G}^{\rd}$.
Likewise, $\wk_y = \wk_y^{(\graph{G})}$ goes from $y$ to $x$ and then follows
$\wk_x^{\rd}$.  The rest of the trails of \graph{G} are the same as those of 
$\graph{G}^{\rd}$.  Thus \edge{1} is used exactly twice, as are the rest of the edges.
\end{proof}

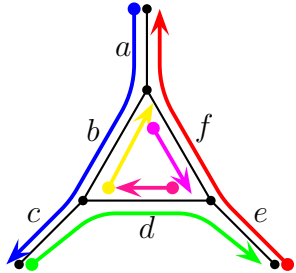
\begin{figure}
\begin{center}
{\psset{unit=1.7,arrowsize=.15}
\begin{pspicture}(0,-.5)(1,2.1)
  \rput(0,0){\rnode{n1}{\psdot(0,0)}}
  \rput(.5,0.866025403784439){\rnode{n2}{\psdot(0,0)}}    
  \rput(1,0){\rnode{n3}{\psdot(0,0)}}    
  \rput(.5,1.5){\rnode{n4}{\psdot(0,0)}}    
  \rput(-.5,-.5){\rnode{n5}{\psdot(0,0)}}    
  \rput(1.5,-.5){\rnode{n6}{\psdot(0,0)}}    
  \ncline{n1}{n2}
  \naput{$b$}
  \ncline{n2}{n3}
  \naput{$f$}
  \ncline{n1}{n3}
  \nbput{$d$}
  \ncline{n4}{n2}
  \nbput{$a$}
  \ncline{n5}{n1}
  \naput{$c$}
  \ncline{n6}{n3}
  \nbput{$e$}
  \psline[linewidth=.03,linecolor=blue,linearc=.75]{*->}(.4,1.5)(.4,0.866025403784439)(-.1,0)(-.6,-.5)
  \psline[linewidth=.03,linecolor=red,linearc=.75]{<-*}(.6,1.5)(.6,0.866025403784439)(1.1,0)(1.6,-.5)
  \psline[linewidth=.03,linecolor=green,linearc=.45]{*->}(-.4,-.5)(.1,-.1)(.9,-.1)(1.4,-.5)
  \psline[linewidth=.03,linecolor=yellow]{*->}(.2,.1)(.55,0.766025403784439)
  \psline[linewidth=.03,linecolor=magenta]{*->}(.55,0.566025403784439)(.85,.05)
  \psline[linewidth=.03,linecolor=deeppink]{*->}(.7,.1)(.25,.1)
\end{pspicture}}
\end{center}
\caption{A Trail Double Cover That Is Not Realizable.}
\label{trailSpace}
\end{figure}

\begin{defn}
A Trail Double Cover \trails{T} is \dword{realizable} if there is an edge labeling of the graph such that 
the resulting set of \migt \ trails is \trails{T}.
\end{defn}
By definition, the Trail Double Cover pictured in Figure~\ref{migtGraph} is realizable.
However the Trail Double Cover pictured in Figure~\ref{trailSpace} is 
{\bf not} realizable; this point can be checked directly, however, we will
see that this follows from Theorem~\ref{thm_realize}.
Though the Trail Double Cover of Figure~\ref{trailSpace} is not realizable, we can still see any Trail Double Cover as 
representing a permutation of its vertices: each trail maps its start vertex to its end vertex
(we understand an isolated vertex $v$ to have an associated trivial trail \seq{v} that starts and ends at $v$).
We can view a Trail Double Cover in this way because a unique trail begins at each vertex, and as the next lemma shows, a unique trails ends at each vertex.

\begin{lem}
In any Trail Double Cover, each vertex of the graph is the final vertex of a unique trail.
\end{lem}
\begin{proof}
If the graph has $n$ vertices, then the Trail Double Cover has $n$ trails that start at those
$n$ vertices. So it suffices to show that each vertex is the final vertex of some
trail.  Suppose, for contradiction, that there were a vertex $v$ which was not the final
vertex of any trail.  Then the trail that starts at $v$ contains an odd number
of edges incident to $v$, and any other trail contains an even number of edges
incident to $v$. This implies that in total, an odd number of edges incident to $v$ are in use by some trail, however, this contradicts the property that  
in a Trail Double Cover, each edge incident to $v$ is in use by exactly two
trails.
\end{proof}

We give a characterization of realizability using an auxiliary \emph{digraph} (i.e. directed graph).
Edges in a digraph are called \dword{arcs}; we will refer to the arc that starts at vertex $x$ and goes to vertex $y$ by \arc{x,y}.  In a digraph, we refer to a \dword{directed trail} as a trail that always moves in the direction of the arcs.

From a Trail Double Cover we will create an auxiliary digraph by converting each one of the trails in the Trail Double Cover into a directed trail in a new digraph.
\begin{defn}

Given any Trail Double Cover \trails{T} on \graph{G}, we define its \dword{Edge Digraph} to be the following digraph:

\begin{itemize}

\item Its vertices are the edges of \graph{G}.

\item 
For each trail
\walk{v_1, e_1, v_2, e_2, \ldots, v_k, e_k, v_{k+1}} in \trails{T}, we have the following arcs:
$\arc{e_1, e_2}, \arc{e_2, e_3}, \ldots, \arc{e_{k-1}, e_k}$.

\end{itemize}

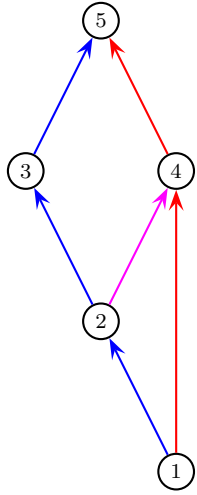
\begin{figure}[htbp]
  \centering
\psset{unit=2,arrowsize=.1}
  \begin{pspicture}(0,-.3)(1,3.2)
    \rput(1,0){\circlenode{1}{\tiny $1$}}
    \rput(.5,1){\circlenode{2}{\tiny $2$}}
    \rput(0,2){\circlenode{3}{\tiny $3$}}
    \rput(1,2){\circlenode{4}{\tiny $4$}}
    \rput(.5,3){\circlenode{5}{\tiny $5$}}
    \ncline[linecolor=blue]{->}{1}{2}
    \ncline[linecolor=blue]{->}{2}{3}
    \ncline[linecolor=blue]{->}{3}{5}
    \ncline[linecolor=red]{->}{1}{4}
    \ncline[linecolor=red]{->}{4}{5}
    \ncline[linecolor=magenta]{->}{2}{4}
  \end{pspicture}
  \caption{The Edge Digraph Of The Trail Double Cover In Figure~\ref{migtGraph}}
  \label{fig:dagofexmpl}
\end{figure}

\end{defn}
For example, Figure~\ref{digraph} is the Edge Digraph of the graph in 
Figure~\ref{trailSpace}, and Figure~\ref{fig:dagofexmpl} is the Edge Digraph
of the Trail Double Cover in Figure~\ref{migtGraph}.  The next theorem gives us a criterion
for determining whether or not a Trail Double Cover is realizable.
Applying the theorem to the graph of Figure~\ref{trailSpace}, we see that 
it is {\bf not} realizable, since its Edge Digraph, drawn in 
Figure~\ref{digraph}, {\bf does} have a directed cycle.

\begin{figure}
\begin{center}
{\psset{unit=1.5,arrowsize=.15}
  \begin{pspicture}(-1.3,-1.3)(1.3,1.3)
    \rput(1.00000000000000, 0.000000000000000){\rnode{d}{\psdot(0,0)}}    
    \uput[0](1.00000000000000,0.000000000000000){$d$}
    \rput(0.500000000000000, 0.866025403784439){\rnode{c}{\psdot(0,0)}}    
    \uput[45](0.500000000000000, 0.866025403784439){$c$}
    \rput(-0.500000000000000, 0.866025403784439){\rnode{b}{\psdot(0,0)}}    
    \uput[135](-0.500000000000000, 0.866025403784439){$b$}
    \rput(-1.00000000000000, 0.000000000000000){\rnode{a}{\psdot(0,0)}}    
    \uput[180](-1.00000000000000, 0.000000000000000){$a$}
    \rput(-0.500000000000000, -0.866025403784439){\rnode{f}{\psdot(0,0)}}    
    \uput[-135](-0.500000000000000, -0.866025403784439){$f$}
    \rput(0.500000000000000, -0.866025403784439){\rnode{e}{\psdot(0,0) }}    
    \uput[-45](0.500000000000000, -0.866025403784439){$e$}
    \ncline[linecolor=blue]{->}{a}{b}
    \ncline[linecolor=blue]{->}{b}{c}
    \ncline[linecolor=green]{->}{c}{d}
    \ncline[linecolor=green]{->}{d}{e}
    \ncline[linecolor=red]{->}{e}{f}
    \ncline[linecolor=red]{->}{f}{a}
  \end{pspicture}}
\end{center}
\caption{The Edge Digraph Of The Trail Double Cover In Figure~\ref{trailSpace}.}
\label{digraph}
\end{figure}
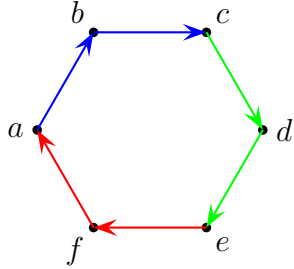

\begin{thm} \label{thm_realize}
A Trail Double Cover is realizable if and only if
its Edge Digraph has {\bf no} directed cycle.

\end{thm}
\begin{proof}

Suppose \graph{G} is the graph and \trails{T} is a Trail Double Cover on it.
Let \graph{D} be the Edge Digraph of \trails{T}.

{\bf Forward Direction}: Consider an edge labeling that yields \trails{T}.
Any arc of \graph{D} goes from $e_1$ to $e_2$, where $e_1$ and $e_2$ are edges
of \graph{G} such that the label of $e_1$ is less than the label of $e_2$.
So \graph{D} cannot have a directed cycle.

{\bf Backwards Direction}:  Supposing \graph{D} has no directed cycle, take
 any topological sort of \graph{D} to arrive at an edge labeling of \graph{G}.  
For example, the graph of Figure~\ref{migtGraph} has the edge digraph in Figure~\ref{fig:dagofexmpl}, so one topological sort is 1,2,3,4,5, which corresponds to the original edge labeling of the graph in Figure~\ref{migtGraph}, while the other topological sort is 1,2,4,3,5, giving a different edge labeling of the graph in Figure~\ref{migtGraph}.
Let $\graph{G}^*$ be \graph{G} with its edges labeled according to the topological sort.   
We show that the set of  \migt{s} of $\graph{G}^*$ is exactly \trails{T}. 
For any Trail Double Cover, at each vertex $x$ we get the situation shown in 
Figure~\ref{figActionAtVertex} (left side): One trail, say $\wk_1$, starts at $x$ by using
edge $e_1$, one trail, say $\wk_d$, ends at $x$ by using edge $e_d$,
and the rest of the trails enter $x$ by using one edge and leave by another. We 
let $\wk_{(i, i+1)}$ refer to the trail that enters along $e_i$ and leaves along
$e_{i+1}$, for $i = 1, 2, \ldots, d-1$.  It is possible that some of the $d+1$ trails
$\wk_1, \wk_d, \wk_{(1,2)}, \wk_{(2,3)}, \ldots, \wk_{(d-1, d)}$ are not distinct,
and any $e_i$ and $e_j$ might be multi-edges attaching the same two vertices.
The related part of \graph{D} is pictured 
in Figure~\ref{figActionAtVertex} (right side).  In our topological sort we 
must have $e_1 < e_2 < \cdots < e_d$.  Now, consider any trail 
$\wk = \seq{v_1, w_1, v_2, \ldots, v_{k-1}, w_{k-1}, v_k}$ in \trails{T};
we show that in the topological ordering of the edges, \wk \ is an \migt \ in $\graph{G}^*$.
Referring to Figure~\ref{figActionAtVertex} (left side), we can see that edge $w_1$
of \wk, being its first edge corresponds to an edge like $e_1$ of  
Figure~\ref{figActionAtVertex} (right side)
and so it is the smallest edge incident to $v_1$, as required.  Similarly, edge $w_{k-1}$ of  \wk \ corresponds to edge like
$e_d$, and so it is the largest labeled edge at $v_k$, as required.
Consider any intermediate edges $w_i$ and $w_{i+1}$ of \wk,
both incident to vertex $v_{i+1}$; these edges 
correspond to some
$e_j$ and $e_{j+1}$ in Figure~\ref{figActionAtVertex}.  We suppose the trail has been \migt \ up to and including $w_i$,
and show that it still is for $w_{i+1}$.  The topological ordering of the edges
pictured in Figure~\ref{figActionAtVertex} (right side) indicates that 
$w_{i+1}$ is the next largest labeled edge, so this matches the \migt.
Thus we have shown that every trail of \trails{T} is an \migt \ in $\graph{G}^*$.
Also note that every \migt \ of $\graph{G}^*$ is in \trails{T} since being a
Trail Double Cover, \trails{T} used every edge twice, so there can be no
more \migt{s}.  Thus \trails{T} is exactly the set of \migt{s} on $\graph{G}^*$,
so we have realized \trails{T}.  
\end{proof}

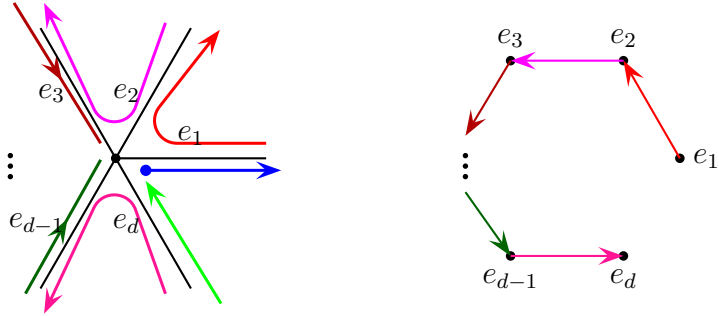
\begin{figure}
\begin{center}
{\psset{unit=2,arrowsize=.12}
\begin{pspicture}(-1.5,-1.2)(7,1.2)
\rput(-.7,0){\huge \vdots}
\psdot(0,0)
\psline[linewidth=.02,linecolor=red,linearc=.15]{->}(1.,.1)(.1,.1)(0.692252430155,0.854860557603)
\psline[linewidth=.02,linecolor=magenta,linearc=.15]{->}(0.330934079986,0.901067248718)(.01,0)(-0.476222157659,1.03366188286)
\psline[linewidth=.02,linecolor=deeppink,linearc=.15]{->}(0.330934079986,-0.901067248718)(.01,0)(-0.476222157659,-1.03366188286)
\psline[linewidth=.02,ArrowInside=->,linecolor=lightred](-0.676222157659,1.03366188286)(-.1,.1)
\psline[linewidth=.02,ArrowInside=->,linecolor=darkgreen](-0.6,-.9)(-.1,-.01)
\psline[linewidth=.02,linecolor=blue]{*->}(0.2,-.08)(1.1,-0.08)
\psline[linewidth=.02,linecolor=green]{->}(0.700000000000000, -0.966025403784439)(.2,-.15)
\psline(0,0)(1.00000000000000, 0.000000000000000)	   
\psline(0,0)(0.500000000000000, 0.866025403784439)   
\psline(0,0)(-0.500000000000000, 0.866025403784439)  
\psline(0,0)(-0.500000000000000, -0.866025403784439) 
\psline(0,0)(0.500000000000000, -0.866025403784439)
\uput[90](0.500000000000000, 0.000000000000000){$e_1$}
\uput[180](0.250000000000000, 0.433012701892219){$e_2$}
\uput[180](-0.250000000000000, 0.433012701892219){$e_3$}
\uput[180](-0.250000000000000, -0.433012701892219){$e_{d-1}$}
\uput[180](0.250000000000000, -0.433012701892219){$e_d$}
\rput(3,0){
\psset{unit=1.5}
\rput(0.500000000000000, 0.000000000000000){\rnode{1}{\psdot(0,0)}}    
\uput[0](0.500000000000000, 0.000000000000000){$e_1$}
\rput(0.250000000000000, 0.433012701892219){\rnode{2}{\psdot(0,0)}} 
\uput[90](0.250000000000000, 0.433012701892219){$e_2$}
\rput(-0.250000000000000, 0.433012701892219){\rnode{3}{\psdot(0,0)}}    
\uput[90](-0.250000000000000, 0.433012701892219){$e_3$}
\rput(-0.250000000000000, -0.433012701892219){\rnode{5}{\psdot(0,0)}}
\uput[-90](-0.250000000000000, -0.433012701892219){$e_{d-1}$}
\rput(0.250000000000000, -0.433012701892219){\rnode{6}{\psdot(0,0)}}    
\uput[-90](0.250000000000000, -0.433012701892219){$e_d$}
\ncline[linecolor=red]{->}{1}{2}
\ncline[linecolor=magenta]{->}{2}{3}
\ncline[linecolor=deeppink]{->}{5}{6}
\rput(-.45,0){\huge \vdots}
\psline[linecolor=lightred]{->}(-0.250000000000000, 0.433012701892219)(-.45,.1)
\psline[linecolor=darkgreen]{->}(-.45,-.15)(-0.250000000000000, -0.433012701892219)
}
\end{pspicture}}
\end{center}
\caption{The Configuration Of Trails In The Neighborhood Of A Vertex.}
\label{figActionAtVertex}
\end{figure}

\section{Dual Graph}
\label{sec:greedy}

In this section we define the notion of the dual of a transposition sequence via the graph interpretation.
This notion of dual will turn out to be equivalent to the Mind-Body Dual.

\begin{defn}
Given a labeled graph \graph{G}, by $\graph{G'}$, the \dword{Trail Dual}, we mean the following labeled graph:
\begin{itemize}

\item
The vertices of $\graph{G}'$ are the same as those of \graph{G}.

\item
The edges of $\graph{G}'$ are determined as follows: 
For any vertices $x$ and $y$, if the edge labeled $k$ is used by
both trail $\wk_x$ and trail $\wk_y$, then make an edge labeled $k$ between $x$ and $y$.
\end{itemize}

\end{defn}
For example, consider the graph of Figure~\ref{migtGraph}, with its trails displayed.  Since $\wk_1$ and $\wk_3$ both use edge $4$, the Trail Dual will
have an edge with label $4$ between vertices 1 and 3; the full Trail Dual is shown in Figure~\ref{dualGraph}.
Notice that the Mind-Body Dual from Example~\ref{ex_dual}
viewed as a graph is exactly the graph in Figure~\ref{dualGraph}. 
In the next theorem we point out that this is always true.

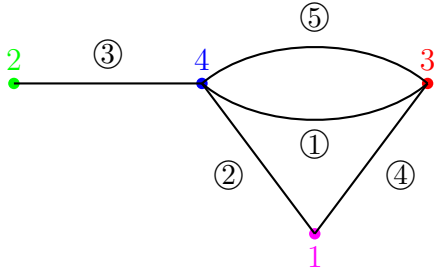
\begin{figure}
\begin{center}
\begin{pspicture}(-.3,1.5)(6.5,5.2)
  \rput(0,4){\rnode{n2}{\psdot[linecolor=green,dotscale=1.2](0,0)}}  
  \uput[90](0,4){\green $2$}
  \rput(2.5,4){\rnode{n4}{\psdot[linecolor=blue,dotscale=1.2](0,0)}}  
  \uput[90](2.5,4){\blue $4$}
  \rput(5.5,4){\rnode{n3}{\psdot[linecolor=red,dotscale=1.2](0,0)}}  
  \uput[90](5.5,4){\red $3$}
  \rput(4,2){\rnode{n1}{\psdot[linecolor=magenta,dotscale=1.2](0,0)}}  
  \uput[-90](4,2){\magenta $1$}
  \ncline{n2}{n4}
  \naput{\raisebox{.5pt}{\textcircled{\raisebox{-.9pt} {\small $3$}}}}
  \ncline{n1}{n4}
  \naput{\raisebox{.5pt}{\textcircled{\raisebox{-.9pt} {\small $2$}}}}
  \ncline{n1}{n3}
  \nbput{\raisebox{.5pt}{\textcircled{\raisebox{-.9pt} {\small $4$}}}}
  \ncarc[arcangle=40]{n3}{n4}
  \naput{\raisebox{.5pt}{\textcircled{\raisebox{-.9pt} {\small $1$}}}}
  \ncarc[arcangle=40]{n4}{n3}
  \naput{\raisebox{.5pt}{\textcircled{\raisebox{-.9pt} {\small $5$}}}}
\end{pspicture}
\end{center}
\caption{Trail Dual Of The Graph From Figure~\ref{migtGraph}.}
\label{dualGraph}
\end{figure}

\begin{thm} \label{thmTrailDualChar}
For any transposition sequence, its Mind-Body Dual is the same as its Trail Dual.
\end{thm}
\begin{proof}

Suppose the transposition sequence is \seq{\per{s_1}, \ldots, \per{s_m}}, with Mind-Body Dual
\seq{\per{s_1'}, \ldots, \per{s_m'}}, and Trail Dual
\seq{\per{t_1}, \ldots, \per{t_m}}.
We show that for any $k$, $\per{t_k} = \per{s_k'}$.
Suppose that $A$ is the Mind-Body Assignment 
$\id \bact \seq{\per{s_1}, \ldots, \per{s_{k-1}}}$ and $\per{s_k} = (x,y)$.  Thus
 $\per{s_k'} = \mb_{A}(\per{s_k}) = (x',y')$, i.e. in $A$, $x'$ is above $x$, and $y'$ is above $y$.  Note that 
$\id \bact \seq{\per{s_1}, \ldots, \per{s_{k-1}}} = \per{s_1} \cdots \per{s_{k-1}}$ (by Lemma~\ref{lem_acts_are_products}).  
Consider the \migt{s} of the labeled graph \seq{\per{s_1}, \ldots, \per{s_{k-1}}}; by Lemma~\ref{lem_trail_trajectory}, $\wk_{x'}$ starts at $x'$ and ends at $x$, and  $\wk_{y'}$ starts at $y'$ and ends at $y$.  Thus when edge \per{s_k} is added to the graph, $\wk_{x'}$ is extended by moving along the edge labeled $k$ to now end at $y$, while $\wk_{y'}$ is extended by moving along the edge labeled $k$ to now end at $x$.  So $\per{t_k} = (x', y') = \per{s'_k}$, and we are done.
\end{proof}

Now that we know that the Mind-Body Dual and the Trail Dual are the same, we can make a further
connection between the graph interpretation and the mind-body interpretation; also we may use the term \dword{dual} to refer to either of the equivalent notions.
  Recall Lemma~\ref{lem_actions_and_dual} and the discussion that
follows it.  From that discussion we can conclude that when a transposition sequence is displayed as a labeled graph, as in 
Figure~\ref{graphOfTranSeq}, this shows the sequence of bodies that will step into
The Machine; in this case, first $3$ and $4$, followed by $1$ and $3$, and so on.
Now we know that the Trail Dual, as in Figure~\ref{dualGraph}, shows the corresponding sequence of minds that step into the machine;
in this case, first $3$ and $4$, followed by $1$ and $4$, and so on.
Furthermore, two trails $\wk_x$ and $\wk_y$ of the original graph intersecting at the edge labeled $k$ means that on the
$k^{th}$ swap, minds $x$ and $y$ step into The Machine;
i.e. to see which non-visible minds go into The Machine, just look at where 
the trails cross.

\section{Algebraic Characterization of The Dual}
\label{sec_alg}

This section provides an algebraic characterization of the dual, which leads to a 
a simple graph algorithm for computing the dual.

\begin{defn} Suppose \per{p} and \per{t} are permutations, and
\seq{\per{s_1}, \ldots, \per{s_k}} is a transposition sequence.

\begin{itemize}

\item
Let $\per{p}^{\per{t}}$ be
the conjugate $\per{t}^{-1} \per{p} \per{t}$.  

\item
Let $\seq{\per{s_1}, \ldots, \per{s_k}}^{\per{t}}$ be
$\seq{\per{s_1}^{\per{t}}, \ldots, \per{s_k}^{\per{t}}}$.

\end{itemize}

\end{defn}

\begin{lem} \label{lem_permexp}
For any permutations \per{p}, \per{a}, and \per{b} we have
$(\per{p}^{\per{a}})^{\per{b}} = \per{p}^{({\per{a}} \cdot {\per{b}})}$.
\end{lem}

A simple but useful observation is to note that for transpositions
\per{s} and $\per{t} = (x, y)$, we have that $\per{s}^{\per{t}}$ is just
\per{s} with $x$ replaced by $y$ and $y$ replaced by $x$.

\begin{exm}
$$\seq{(3,4), (1,3), (1,2), (3,4), (2,3)}^{(3,4)} = 
\seq{(3,4), (1,4), (1,2), (3,4), (2,4)}$$

\end{exm}

We now state and prove the key technical lemma for this section, before proving Theorem~\ref{thm_algchar}, which characterizes the dual algebraically.
By writing  
$\seq{\per{t}} \seq{\per{s_1}, \ldots, \per{s_m}}$ we mean
\seq{\per{t}, \per{s_1}, \ldots, \per{s_m}}.

\begin{lem} \label{lemDualIncrement}
For transpositions $\per{t}, \per{s_1}, \ldots, \per{s_m}$, we have:
$$\seq{\per{t}, \per{s_1}, \ldots, \per{s_m}}' = 
(\seq{\per{t}}  \seq{\per{s_1}, \ldots, \per{s_m}}')^{\per{t}}.$$

\end{lem}
\begin{proof}

Suppose \seq{A_0, A_1, \ldots, A_m} is the Mind-Body Sequence corresponding to
\seq{\per{s_1}, \ldots, \per{s_m}}.
Now we describe the Mind-Body Sequence corresponding to 
\seq{\per{t}, \per{s_1}, \ldots, \per{s_m}}.
Suppose $\per{t} = (x,y)$, where $x < y$.  Our Mind-Body Sequence starts with \id, and then $\id \bact \per{t}$ yields \id, except that bodies $x$ and $y$ have been swapped; call the resulting Mind-Body Assignment $A_0^*$.  We can write $A_0^*$
as follows (\emph{instead of swapping the bodies, we equivalently swap the minds $x$ and $y$, and leave the bodies in order}): 
$$ \mindbod{1, 2,  \ldots, y, \ldots, x, \ldots n}{1, 2, \ldots, x, \ldots, y, \ldots, n}$$  

\noindent
Thus we can make the following \emph{observation}:  
\begin{quote}
The Mind-Body Sequence corresponding to 
\seq{\per{t}, \per{s_1}, \ldots, \per{s_m}} is
\seq{\id, A_0^*, A_1^*, \ldots, A_m^*}, where $A_i^*$ is just $A_i$ with
minds $x$ and $y$ swapped.  
\end{quote}
Now we can compare the duals
$\seq{\per{t}, \per{s_1}, \ldots, \per{s_m}}' = 
\seq{\per{t}, \per{s^*_1}, \ldots, \per{s^*_m}}$ and
$\seq{\per{s_1}, \ldots, \per{s_m}}' = \seq{\per{s'_1}, \ldots, \per{s'_m}}$.
Consider some $\per{s_k} = (a,b)$.
From the definition of Mind-Body Dual, to determine $\per{s_k}'$ we look at what
is above $a$ and $b$ in $A_{k-1}$, while for \per{s^*_k} we look in
$A_{k-1}^*$, and so by the above observation \per{s^*_k} is just \per{s'_k}
with $x$ and $y$ swapped, i.e. $\per{s_k^*} = (\per{s'_k})^{\per{t}}$.  So the lemma
follows.
\end{proof}

\begin{thm} \label{thm_algchar}
For any transposition sequence \seq{\per{s_1}, \ldots, \per{s_m}}, its dual is

$$\seq{ \per{s_1}, \per{s_2}^{\per{s_1}}, \per{s_3}^{\per{s_2} \per{s_1}}, \ldots, \per{s_m}^{\per{s_{m-1}} \cdots \per{s_1}  } }$$

\end{thm}

\begin{proof}

We proceed by induction on the length of the transposition sequence, showing the inductive step.

\begin{alignat*}{2}
\seq{\per{s_1}, \ldots, \per{s_m}}' 
 &  = (\seq{\per{s_1}}  \seq{\per{s_2}, \ldots, \per{s_m}}')^{\per{s_1}} && \text{, by lemma~\ref{lemDualIncrement}} \\ 
 & = (\seq{\per{s_1}}  \seq{\per{s_2}, \per{s_3}^{\per{s_2}},  \ldots, 
\per{s_m}^{\per{s_{m-1}} \cdots \per{s_2}} })^{\per{s_1}} && \text{, by inductive hypothesis} \\ 
 & = (\seq{\per{s_1}, \per{s_2}, \per{s_3}^{\per{s_2}},  \ldots, 
\per{s_m}^{\per{s_{m-1}} \cdots \per{s_2}} })^{\per{s_1}} && \text{} \\ 
 & = \seq{ \per{s_1}, \per{s_2}^{\per{s_1}}, \per{s_3}^{\per{s_2} \per{s_1}}, \ldots, \per{s_m}^{\per{s_{m-1}} \cdots \per{s_1}  } } && 
\text{, by Lemma~\ref{lem_permexp}}    
\end{alignat*}
\end{proof}

\begin{exm}
Recall that in Example~\ref{ex_trans_seq} we considered the transposition sequence 
\seq{(3,4), (1,3), (1,2), (3,4), (2,3)}.  If its dual is \seq{\per{s'_1}, \per{s'_2}, \per{s'_3}, \per{s'_4},\per{s'_5}}, we can, for example, calculate \per{s'_3}:
$$
\per{s'_3} = \per{s_3}^{s_2 s_1} = (1,2)^{(1,3) (3,4)} = (3,2)^{(3,4)} = (4,2)
$$

\end{exm}

So we can see that the above algebraic characterization of the dual provides a way to think of calculating the dual one edge at a time.
We can also interpret the algebraic characterization as a graph algorithm.
For input, the algorithm takes a 
labeled graph \graph{G} (with $m$ edges) and outputs another labeled graph $\graph{G}^*$
(which will in fact be the dual of \graph{G}).  The algorithm is as follows.

\begin{enumerate}

\item Initialize $\graph{G}^*$ to be the graph with \emph{no edges} and the \emph{same} vertex set as \graph{G}.

\item 
In \graph{G}, proceed from the edge labeled $m$ in order down to the edge labeled $1$; 
for edge $\{a, b\}$, labeled $k$ do the following:

\begin{itemize}

\item
Add an edge labeled $k$ to $\graph{G}^*$ between vertices $a$ and $b$.

\item
Then swap the labels $a$ and $b$ in $\graph{G}^*$.

\end{itemize}

\end{enumerate}
As an example of the graph algorithm, Figure~\ref{graphAlg} shows the algorithm applied to the graph of Figure~\ref{graphOfTranSeq} to obtain the dual of Figure~\ref{dualGraph}.
Now we give the promised alternate algebraic proof of Lemma~\ref{lemDualTwice},
which states that for any transposition sequence \per{s} we have $\per{s''} = \per{s}$.

\begin{proof}

Suppose the $\ps{s} = \seq{\per{s_1}, \ldots, \per{s_m}}$, its dual
$\ps{s'} = \seq{\per{s_1'}, \ldots, \per{s_m'}}$ and the dual of \ps{s'} is
$\ps{s''} = \seq{\per{s_1''}, \ldots, \per{s_m''}}$.
We show that $\per{s_k''} = \per{s_k}$ for any $k = 1, \ldots, m$.
Note that $\per{s''_1} = \per{s'_k} = \per{s_1}$. For $k \ge 2$, we have the following.

\begin{alignat*}{2}
\per{s_k''} 
 & = (\per{s_k'})^{\per{s_{k-1}'} \cdots \per{s_1'}} && \text{, by Theorem~\ref{thm_algchar}} \\ 
 & = (\per{s_k}^{\per{s_{k-1}} \cdots \per{s_1}})^{\per{s_{k-1}'} \cdots \per{s_1'}} && \text{, by Theorem~\ref{thm_algchar} } \\ 
 & = \per{s_k}^{\per{s_{k-1}} \cdots \per{s_1} \cdot \per{s_{k-1}'} \cdots \per{s_1'}} && \text{, by Lemma~\ref{lem_permexp}} 
\end{alignat*}

\noindent
Thus, it suffices to show that 
$\per{s_{k-1}} \cdots \per{s_1} \cdot \per{s_{k-1}'} \cdots \per{s_1'} = \id$, which we prove by induction on $k$.
The base case is true since $\per{s_1'} = \per{s_1}$.  Now we show the inductive step.

\begin{alignat*}{2}
\per{s_k} \cdots \per{s_1} \cdot \per{s_k'} \cdots \per{s_1'} 
 & = \per{s_k} \cdots \per{s_1}
(\per{s_1} \cdots \per{s_{k-1}} \per{s_k} \per{s_{k-1}} \cdots \per{s_1})
\per{s_{k-1}'} \cdots \per{s_1'} && \text{, by Theorem~\ref{thm_algchar} } \\ 
 & = \per{s_{k-1}} \cdots \per{s_1} \cdot \per{s_{k-1}'} \cdots \per{s_1'} && \text{} \\ 
 & = \id && \text{, by inductive hypothesis}  
\end{alignat*}
\end{proof}

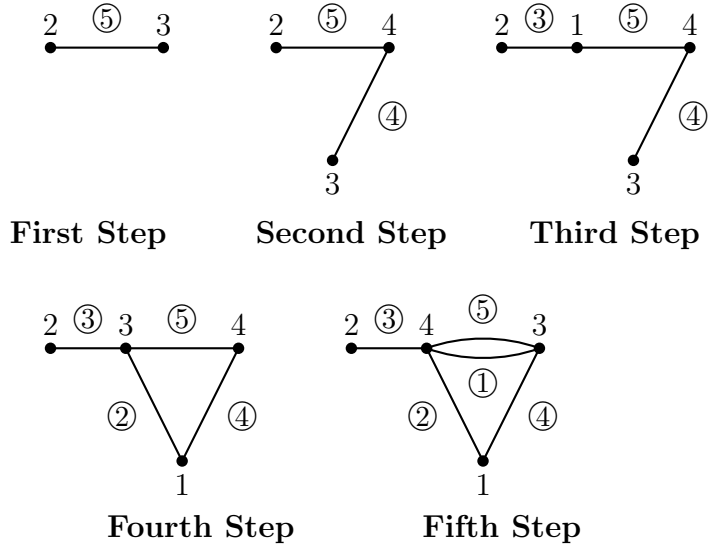
\begin{figure}
\begin{center}
\begin{pspicture}(-1,1.9)(10,9)
  \rput(0,8.5){
  \rput(0,0){\rnode{n2}{\psdot[dotscale=1.2](0,0)}}  
  \uput[90](0,0){$2$}
  \rput(1.5,0){\rnode{n3}{\psdot[dotscale=1.2](0,0)}}  
  \uput[90](1.5,0){$3$}
  \ncline{n2}{n3}
  \naput{\raisebox{.5pt}{\textcircled{\raisebox{-.9pt} {\small $5$}}}}
}
\rput(.5,6){\textbf{First Step}}
  \rput(3,8.5){
  \rput(0,0){\rnode{n2}{\psdot[dotscale=1.2](0,0)}}  
  \uput[90](0,0){$2$}
  \rput(1.5,0){\rnode{n3}{\psdot[dotscale=1.2](0,0)}}  
  \uput[90](1.5,0){$4$}

  \rput(.75,-1.5){\rnode{n4}{\psdot[dotscale=1.2](0,0)}}  
  \uput[-90](.75,-1.5){$3$}

  \ncline{n2}{n3}
  \naput{\raisebox{.5pt}{\textcircled{\raisebox{-.9pt} {\small $5$}}}}
  \ncline{n3}{n4}
  \naput{\raisebox{.5pt}{\textcircled{\raisebox{-.9pt} {\small $4$}}}}
}
\rput(4,6){\textbf{Second Step}}
  \rput(7,8.5){
    \rput(-1,0){\rnode{n1}{\psdot[dotscale=1.2](0,0)}}  
    \uput[90](-1,0){$2$}
  \rput(0,0){\rnode{n2}{\psdot[dotscale=1.2](0,0)}}  
  \uput[90](0,0){$1$}
  \rput(1.5,0){\rnode{n3}{\psdot[dotscale=1.2](0,0)}}  
  \uput[90](1.5,0){$4$}

  \rput(.75,-1.5){\rnode{n4}{\psdot[dotscale=1.2](0,0)}}  
  \uput[-90](.75,-1.5){$3$}

  \ncline{n2}{n3}
  \naput{\raisebox{.5pt}{\textcircled{\raisebox{-.9pt} {\small $5$}}}}
  \ncline{n3}{n4}
  \naput{\raisebox{.5pt}{\textcircled{\raisebox{-.9pt} {\small $4$}}}}
  \ncline{n1}{n2}
  \naput{\raisebox{.5pt}{\textcircled{\raisebox{-.9pt} {\small $3$}}}}
}
\rput(7.5,6){\textbf{Third Step}}
  \rput(1,4.5){
    \rput(-1,0){\rnode{n1}{\psdot[dotscale=1.2](0,0)}}  
    \uput[90](-1,0){$2$}
  \rput(0,0){\rnode{n2}{\psdot[dotscale=1.2](0,0)}}  
  \uput[90](0,0){$3$}
  \rput(1.5,0){\rnode{n3}{\psdot[dotscale=1.2](0,0)}}  
  \uput[90](1.5,0){$4$}

  \rput(.75,-1.5){\rnode{n4}{\psdot[dotscale=1.2](0,0)}}  
  \uput[-90](.75,-1.5){$1$}

  \ncline{n2}{n3}
  \naput{\raisebox{.5pt}{\textcircled{\raisebox{-.9pt} {\small $5$}}}}
  \ncline{n3}{n4}
  \naput{\raisebox{.5pt}{\textcircled{\raisebox{-.9pt} {\small $4$}}}}
  \ncline{n1}{n2}
  \naput{\raisebox{.5pt}{\textcircled{\raisebox{-.9pt} {\small $3$}}}}
  \ncline{n2}{n4}
  \nbput{\raisebox{.5pt}{\textcircled{\raisebox{-.9pt} {\small $2$}}}}
}
\rput(2,2.1){\textbf{Fourth Step}}
  \rput(5,4.5){
    \rput(-1,0){\rnode{n1}{\psdot[dotscale=1.2](0,0)}}  
    \uput[90](-1,0){$2$}
  \rput(0,0){\rnode{n2}{\psdot[dotscale=1.2](0,0)}}  
  \uput[90](0,0){$4$}
  \rput(1.5,0){\rnode{n3}{\psdot[dotscale=1.2](0,0)}}  
  \uput[90](1.5,0){$3$}

  \rput(.75,-1.5){\rnode{n4}{\psdot[dotscale=1.2](0,0)}}  
  \uput[-90](.75,-1.5){$1$}

  \ncarc[arcangle=20]{n3}{n2}
  \naput{\raisebox{.5pt}{\textcircled{\raisebox{-.9pt} {\small $1$}}}}
  \ncline{n3}{n4}
  \naput{\raisebox{.5pt}{\textcircled{\raisebox{-.9pt} {\small $4$}}}}
  \ncline{n1}{n2}
  \naput{\raisebox{.5pt}{\textcircled{\raisebox{-.9pt} {\small $3$}}}}
  \ncline{n2}{n4}
  \nbput{\raisebox{.5pt}{\textcircled{\raisebox{-.9pt} {\small $2$}}}}
  \ncarc[arcangle=20]{n2}{n3}
  \naput{\raisebox{.5pt}{\textcircled{\raisebox{-.9pt} {\small $5$}}}}
}
\rput(6,2.1){\textbf{Fifth Step}}
\end{pspicture}
\end{center}
\caption{Using The Graph Algorithm To Find The Dual.}
\label{graphAlg}
\end{figure}

\section{Alternate Proof For Goulden/Yong Bijection}
\label{sec_alt}

In this section we use our framework to provide a bijection between the vertex labeled
trees and the factorizations of $(n, \ldots, 2,1)$ into $n-1$ transpositions; the bijection enjoys the same properties as the bijection from \cite{GouldenYong}.  In the first subsection we define the function
and show that it is a bijection; in the second subsection, we define and prove
structural properties possessed by this bijection.
In Section~\ref{sec_gydual}, we will show that our bijection is in fact the same 
as the bijection of \cite{GouldenYong}.

\subsection{The Bijection}

We define our bijection as a composition of two functions:
the dual and a (to-be-defined) relabeling function.
In Definition~\ref{def_FT} we defined \Fdown{n}; we now define
\Fup{n} to be the set of length $n-1$ transposition sequences over $\sym_n$ with product $(1,2, \ldots, n)$.
Immediate from D{\'e}nes~\cite{Denes1959} we have the following theorem; the coherence of the subsequent definitions and discussion depends on this fact.
\begin{thm} \cite{Denes1959} \label{thmDenes}
The graphs in \Fup{n} and \Fdown{n} are trees.

\end{thm}

\noindent
For example, 
\seq{(2,3), (4,5), (3,6), (3,5), (1,6), (6,8), (8,9), (6,7)} in \Fdown{9}  is the tree of Figure~\ref{GY_tree} (ignoring the {\migt}s for now).

\begin{defn} \

\begin{itemize}

\item
$\dual: \Fdown{n} \rightarrow \Fup{n}$, takes an input
transposition sequence to its dual transposition sequence.

\vspace{3mm}

\item
$\slide:  \Fup{n} \rightarrow \VT_n$ is the function defined as follows:

\noindent
We begin with a labeled tree (i.e. vertices and edges are labeled).
Keep vertex ``$1$'' labeled ``$1$''.
For every other vertex $v$, let $e_v$ be the edge incident to $v$ that is closest to vertex ``1''.  Label $v$ by $1 + w$, where $w$ is the label on edge $e_v$.
After relabeling all the vertices, erase the edge labels.

\vspace{3mm}

\item
Our desired bijection $\bij: \Fdown{n} \rightarrow \VT_n$ is defined by:
$$\bij = \slide \circ \dual$$

\end{itemize}

\end{defn}

\noindent
See Figure~\ref{fig:slideex} for an example of the \slide \ function applied to the tree
$\seq{(4,5),(3,5),(5,6),(2,8),(2,7),(1,8),(2,6)}$, 
whose product is $(8,7, \ldots, 1)$.  
For an example of the entire bijection \bij, see Figure 2 of \cite{GouldenYong}.

\begin{figure}[htbp]
  \centering
  {\psset{unit=3.4}
\begin{pspicture}(1.2,0)(5.2,1.35)
  \rput(1.1,1.2){\rnode{1}{\psdot[dotscale=1.2](0,0)}}
    \rput(1.7,.7){\rnode{2}{\psdot[dotscale=1.2](0,0)}}
    \rput(2.7,.8){\rnode{3}{\psdot[dotscale=1.2](0,0)}}
    \rput(2.8,.3){\rnode{4}{\psdot[dotscale=1.2](0,0)}}
    \rput(2.55,.55){\rnode{5}{\psdot[dotscale=1.2](0,0)}}
    \rput(2.1,.6){\rnode{6}{\psdot[dotscale=1.2](0,0)}}
    \rput(1.6,.4){\rnode{7}{\psdot[dotscale=1.2](0,0)}}
    \rput(1.35,.9){\rnode{8}{\psdot[dotscale=1.2](0,0)}}
    \uput[45](1.35,.9){$8$}
    \uput[90](1.1,1.2){$1$}
    \uput[90](1.7,.7){$2$}
    \uput[90](2.7,.8){$3$}
    \uput[0](2.8,.3){$4$}
    \uput[120](2.55,.55){$5$}
    \uput[90](2.1,.6){$6$}
    \uput[-90](1.6,.4){$7$}
    \ncline{4}{5}
    \naput{\raisebox{.5pt}{\textcircled{\raisebox{-.9pt} {\small $1$}}}}
    \ncline{3}{5}
    \naput{\raisebox{.5pt}{\textcircled{\raisebox{-.9pt} {\small $2$}}}}
    \ncline{5}{6}
    \nbput{\raisebox{.5pt}{\textcircled{\raisebox{-.9pt} {\small $3$}}}}
    \ncline{2}{8}
    \nbput{\raisebox{.5pt}{\textcircled{\raisebox{-.9pt} {\small $4$}}}}
    \ncline{2}{7}
    \nbput{\raisebox{.5pt}{\textcircled{\raisebox{-.9pt} {\small $5$}}}}
    \ncline{1}{8}
    \naput{\raisebox{.5pt}{\textcircled{\raisebox{-.9pt} {\small $6$}}}}
    \ncline{2}{6}
    \naput{\raisebox{.5pt}{\textcircled{\raisebox{-.9pt} {\small $7$}}}}
\uput[-90](2.1,.25){$T$}
\rput(3.81,.74){
  \begin{pspicture}(3,1.5)
      \rput(1.1,1.2){\rnode{1}{\psdot[dotscale=1.2](0,0)}}
    \rput(1.7,.7){\rnode{2}{\psdot[dotscale=1.2](0,0)}}
    \rput(2.7,.8){\rnode{3}{\psdot[dotscale=1.2](0,0)}}
    \rput(2.8,.3){\rnode{4}{\psdot[dotscale=1.2](0,0)}}
    \rput(2.55,.55){\rnode{5}{\psdot[dotscale=1.2](0,0)}}
    \rput(2.1,.6){\rnode{6}{\psdot[dotscale=1.2](0,0)}}
    \rput(1.6,.4){\rnode{7}{\psdot[dotscale=1.2](0,0)}}
    \rput(1.35,.9){\rnode{8}{\psdot[dotscale=1.2](0,0)}}
    \uput[45](1.35,.9){$7$}
    \uput[90](1.1,1.2){$1$}
    \uput[90](1.7,.7){$5$}
    \uput[90](2.7,.8){$3$}
    \uput[0](2.8,.3){$2$}
    \uput[120](2.55,.55){$4$}
    \uput[90](2.1,.6){$8$}
    \uput[-90](1.6,.4){$6$}
    \ncline{4}{5}
    \ncline{3}{5}
    \ncline{5}{6}
    \ncline{2}{8}
    \ncline{2}{7}
    \ncline{1}{8}
    \ncline{2}{6}
  \end{pspicture}}
\uput[-90](4.41,.25){$\slide(T)$}
 \end{pspicture}}
  \caption{Example Of \ \slide}
  \label{fig:slideex}
\end{figure}

Since \dual \ is a bijection (by Lemma~\ref{lemDualTwice}), to show \bij \ is bijection,
it is enough to show that \slide \ is a bijection.

\begin{lem} 
$\slide$ is a bijection.
\end{lem}
\begin{proof}

To see that \slide \ is a bijection we define its inverse function.
We begin with a vertex labeled tree.  For every vertex $v$, except the vertex labeled ``1'', let $e_v$ be the edge incident to $v$ that is closest to ``1''.  If vertex $v$ is labeled by $w$, then label edge $e_v$ by $w-1$.  Erase all the vertex labels except ``1''.  Now label the vertices using the \migt{s}, i.e. follow $\wk_1$ from ``1'' to its final vertex, labeling it ``2'', then follow $\wk_2$ from ``2'' to determine which vertex to label ``3'', and so on.
\end{proof}

As noted by other authors (e.g. \cite{GouldenYong}),
since it is known that $|\VT_n| = n^{n-2}$, the bijection \bij \ shows that
$|\Fdown{n}| = n^{n-2}$.

\subsection{The Structural Property of the Bijection}


Now we review a structural property of the bijection, defined in \cite{GouldenYong}, showing that our bijection has this property.

\begin{defn} 

Suppose the transposition sequence (over $\sym_n$) $\ps{s} = \seq{\per{s_1}, \ldots, \per{s_{n-1}}}$ is a tree, and so for any $\per{s_k} = (x, y)$, we can write the product of \ps{s} as
$(x, x_1, \ldots, x_a, y, y_1, \ldots, y_b)$.

\begin{itemize}

\item 
We let \cpart{\ps{s}}{\per{s_k}} be the partition  
$\{ \ \{x, x_1, \ldots, x_a \} , \{y, y_1, \ldots, y_b \} \ \}$.

\item 
We let \fpart{\ps{s}}{\per{s_k}} be the partition
of the vertices of the tree into two sets:
When we remove the edge \per{s_k} from the tree, we take the vertices 
in each connected component.

\end{itemize}
 
\end{defn}

\begin{defn} 
Suppose that \per{t} is some transposition in the transposition sequence \ps{s},
where \ps{s} is a tree.

\begin{itemize}

\item
$\cnorm(\per{t}) = \hbox{min}(|A|, |B|)$, where \cpart{\ps{s}}{\per{t}} is 
the partition $\{ A, B\}$.

\item
$\fnorm(\per{t}) = \hbox{min}(|A|, |B|)$, where \fpart{\ps{s}}{\per{t}} is 
the partition $\{ A, B\}$.

\end{itemize}

\end{defn}

\noindent
For example, consider Figure~\ref{fig:slideex} and let $\per{t} = (2,8)$.
Since removing edge $\{ 2,8 \}$ from the tree leads to the vertex partition $\{ \ \{1, 8 \}, \ \{ 2,3,4,5,6,7  \} \  \}$, $\fnorm(\per{t}) = 2$.
In the corresponding permutation cycle $(8,7, \ldots, 1)$, the transposition \per{t} creates the partition
$\{ \ \{1, 2 \}, \ \{ 3,4,5,6,7,8  \} \  \}$, so $\cnorm(\per{t}) = 2$.
In \cite{GouldenYong}, the \cnorm \ is  called the \emph{difference index} 
and the \fnorm \ is called the \emph{edge-deletion index}.

Now we show that the bijection has the desired structural property,
by first proving a stronger property of the dual.

\begin{thm} \label{lemma_structural}
Suppose $\ps{s} = \seq{s_1, \ldots, s_{n-1}}$ is a transposition sequence (over $\sym_n$)
which is a tree, and suppose $\ps{s}' = \seq{s'_1, \ldots, s'_{n-1}}$ is its dual.
Then for $k = 1, \ldots, {n-1}$ we have that
$\fpart{\ps{s}}{\per{s_k}} = \cpart{\ps{s'}}{\per{s'_k}}$ and
$\cpart{\ps{s}}{\per{s_k}} = \fpart{\ps{s'}}{\per{s'_k}}$.

\end{thm}

\begin{proof}

We show $\fpart{\ps{s}}{\per{s_k}} = \cpart{\ps{s'}}{\per{s'_k}}$ 
($\cpart{\ps{s}}{\per{s_k}} = \fpart{\ps{s'}}{\per{s'_k}}$ then follows using Lemma~\ref{lemDualTwice}).
We proceed by induction on the length of the transposition sequence,
considering the inductive step.
Consider $\ps{t} = \seq{\per{s_2}, \ldots, \per{s_{n-1}}}$ and suppose $\per{s_1} = (x,y)$.
So \ps{t} is two trees, say $\graph{T}_x$ and $\graph{T}_y$, where $\graph{T}_x$ is the tree
containing $x$ and $\graph{T}_y$ is the tree containing $y$.
The product of $\graph{T}_x$ is some permutation cycle $C_x = (x, x_1, \ldots, x_a)$
and the product of $\graph{T}_y$ is some permutation cycle $C_y = (y, y_1, \ldots, y_b)$.
Thus the product of \ps{s} is $C = (x, y_1, \ldots, y_b, y, x_1, \ldots, x_a)$
and so by Lemma~\ref{lem_inverse}, the product of $\ps{s}'$ is $C' = (x_a, \ldots, x_1, y, y_b, \ldots, y_1, x)$. We now demonstrate that 
$\fpart{\ps{s}}{\per{s_k}} = \cpart{\ps{s'}}{\per{s'_k}}$.

Consider the case in which $\per{s_k} = \per{s_1}$. Then
$\fpart{\ps{s}}{\per{s_1}} =  $

\noindent
$\{ \ \{x, x_1, \ldots, x_a \}, \{y, y_1, \ldots, y_b\} \ \}
= \cpart{\ps{s'}}{\per{s'_1}}$,
where we get the latter equality by noting that $\per{s'_1} = \per{s_1} = (x,y)$ and 
recalling the value of $C'$.

Now we consider the case in which $\per{s_k} \neq \per{s_1}$, and suppose, without loss of generality, that
$\per{s_k}$ is in $\graph{T}_x = \seq{\per{s_{i_1}}, \ldots, \per{s_{i_c}}}$.  
We remark that for $\graph{T}_x$ and \ps{t} we will keep the edge labels coming from the original tree \ps{s} (so for example, edge \per{s_2} in \ps{t} is labeled $2$, not $1$, and \per{s_{i_1}} in $\graph{T}_x$ is labeled $i_1$, not $1$);
all the relevant definitions and facts work in the same manner for such edge labellings.
Let the dual of $\graph{T}_x$ be  
$\graph{T}^*_x = \seq{\per{s^*_{i_1}}, \ldots, \per{s^*_{i_c}}}$, whose product, by Lemma~\ref{lem_inverse}, is the permutation cycle $C'_x = (x_a, \ldots, x_1, x)$.
Suppose the dual of \ps{t} is $\ps{t}^* = \seq{\per{s^*_2}, \ldots, \per{s^*_{n-1}}}$; note that since \ps{t} consists of two disjoint graphs, for any \per{s_k} in $\graph{T}_x$, \per{s_k^*} is indeed the same in $\graph{T}^*_x$ and $\ps{t}^*$.
Suppose $\per{s_k^*} = (x_i, x_j)$, 
where $0 \le i < j$ (understanding $x_0$ to be $x$).
We can conclude that $\fpart{\graph{T}_x}{\per{s_k}} = \cpart{\graph{T}^*_x}{\per{s^*_k}} =
\{ \ \{ x_j, \ldots, x_{i+1} \},  \{ x_i, \ldots, x_1,\emph{x},x_a, \ldots,x_{j+1}  \}   \ \}$,
where the first equality holds by inductive hypothesis and second by definition,
recalling the value of $C'_x$.
Now consider the entire tree \ps{s}, consisting of $\graph{T}_x$ and $\graph{T}_y$ joined by edge 
\per{s_1}.  When edge \per{s_k} is removed from \ps{s}, all the vertices of $\graph{T}_y$ will be in the component with vertex $x$, so to get \fpart{\ps{s}}{\per{s_k}} we just add the vertices of $\graph{T}_y$ to the 
appropriate piece of the above partition,  
so $\fpart{\ps{s}}{\per{s_k}} = 
\{ \ \{ x_j, \ldots, x_{i+1}  \} , 
\{ x_i, \ldots, x_1, x, x_a, \ldots, x_{j+1}, y, y_1, \ldots, y_b \} \ \}$.
We now show that \cpart{\ps{s}'}{\per{s'_k}} is the same partition.
By Lemma~\ref{lemDualIncrement}, $\per{s'_k} = (\per{s^*_k})^{\per{s_1}}$.  So if $x_i \neq x$, then $\per{s'_k} = (x_i, x_j)$, and if $x_i = x$ then $\per{s'_k} = (y, x_j)$.  In either case, recalling that
$$C' = (x_a, \ldots, x_j, \ldots, x_i, \ldots, x_1, y, y_b, \ldots, y_1, x),$$
we see that 
 \cpart{\ps{s'}}{\per{s'_k}} is the same as \fpart{\ps{s}}{\per{s_k}}.
\end{proof}


The bijection $\bij: \Fdown{n} \rightarrow \VT_n$  first takes a transposition sequence 
$\ps{s} = \seq{\per{s_1}, \ldots, \per{s_{n-1}}}$ to its dual
$\ps{s'} = \seq{\per{s'_1}, \ldots, \per{s'_{n-1}}}$.
By Theorem~\ref{lemma_structural}, for every \per{s_i}, the partitions \cpart{\ps{s}}{\per{s_i}} and 
\fpart{\ps{s'}}{\per{s'_i}} are the same, and so immediately,
$\cnorm(\per{s_i})  = \fnorm(\per{s'_i})$.
Then \bij \ just rearranges the labels, so any transposition \per{s_i} in \ps{s}
has a corresponding edge $e_i$ in the tree $\bij(\ps{s})$ such that
$|\cnorm(\per{s_i})| = |\fnorm(e_i)|$; 
technically we defined \fnorm \ for
transposition sequences, i.e. labeled trees, however the same
basic definition works for a tree in $\VT_n$.
Thus we have given an alternative proof of
the following theorem, which was the main result of \cite{GouldenYong}.

\begin{thm}
The function $\bij: \Fdown{n} \rightarrow \VT_n$ is a bijection with the following property:
Suppose $\ps{s} = \seq{\per{s_1}, \ldots, \per{s_{n-1}}}$ is in \Fdown{n}, 
and $\graph{T} = \bij(\ps{s})$, where \graph{T} has edges $\{e_1, \ldots, e_{n-1} \}$.
Then 
$$\{ |\cnorm(\per{s_1})|, \ldots, |\cnorm(\per{s_{n-1}})| \} =
\{ |\fnorm(e_1)|, \ldots, |\fnorm(e_{n-1})| \}.
$$

\end{thm}

\section{Goulden-Yong Dual}
\label{sec_gydual}

In \cite{GouldenYong}, they define a dual that applies only to trees, using topological methods.  When restricted to trees, we show that their dual is
the same as our dual.

\begin{defn} \cite{GouldenYong}
Given a labeled tree \graph{T} on $n$ vertices, its
\dword{Circle Chord Diagram},  is the following structure:
\begin{quote}
A circle, together with $n$ distinct points on the circle, labeled by the numbers $1, \ldots, n$, in the clockwise direction, 
 drawing a chord between $x$ and $y$ if there is an edge between $x$ and $y$ in \graph{T}.  
\end{quote}

\end{defn}
Consider the tree shown in Figure~\ref{GY_tree} (it is the same as the example in
\cite{GouldenYong}); its \migt{s} are drawn in, but can be ignored for now.  The Circle Chord Diagram for the tree of Figure~\ref{GY_tree}
is shown in Figure~\ref{circleChord}.
Notice that the chords in Figure~\ref{circleChord} are \dword{non-crossing},
i.e. any two chords either do not meet, or only meet at a vertex on the circle.  For a tree with $n$ vertices
and non-crossing chords, its Circle Chord Diagram has the following properties:

\begin{itemize}

\item
The $n$ vertices on the circle, break up the circle into \dword{$n$ arcs}, i.e.
the arc between $1$ and $2$ (we call arc 2), the arc between $2$ and $3$ (we call arc 3), and so on, calling the arc between $n$ and $1$, arc 1. 

\item
The chords break up the region inside the circle into $n$ regions, each containing
one of the $n$ arcs; we refer to this region with arc $k$ as \dword{region $k$}.

\end{itemize}

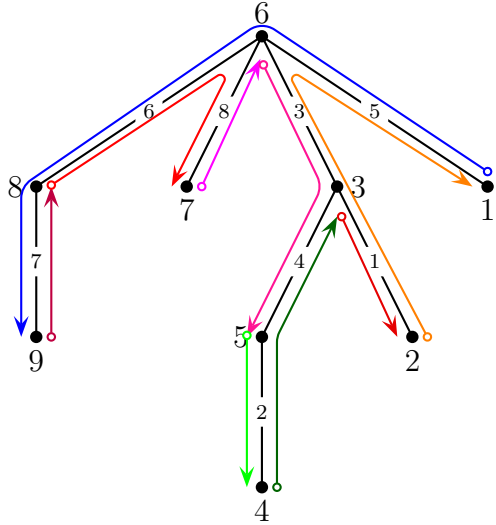
\begin{figure}
\begin{center}
\psset{unit=2,arrowsize=.1}
  \begin{pspicture}(0,.2)(3.2,-3.2)
       \rput(3.0, -1){\rnode{1}{\psdot[dotscale=1.3](0,0)}}
         \uput[-90](3.0, -1){$1$}
       \rput(2.5, -2){\rnode{2}{\psdot[dotscale=1.3](0,0)}}
         \uput[-90](2.5, -2){$2$}
       \rput(2.0, -1){\rnode{3}{\psdot[dotscale=1.3](0,0)}}
         \uput[0](2.0, -1){$3$}
       \rput(1.5, -3){\rnode{4}{\psdot[dotscale=1.3](0,0)}}
         \uput[-90](1.5, -3){$4$}       
       \rput(1.5, -2){\rnode{5}{\psdot[dotscale=1.3](0,0)}}
         \uput[180](1.5, -2){$5$}
       \rput(1.5, 0){\rnode{6}{\psdot[dotscale=1.3](0,0)}}
         \uput[90](1.5, 0){$6$}
       \rput(1.0, -1){\rnode{7}{\psdot[dotscale=1.3](0,0)}}
         \uput[-90](1.0, -1){$7$}
       \rput(0.0, -1){\rnode{8}{\psdot[dotscale=1.3](0,0)}}
         \uput[180](0.0, -1){$8$}
       \rput(0.0, -2){\rnode{9}{\psdot[dotscale=1.3](0,0)}}
         \uput[-90](0.0, -2){$9$}
     \ncline{1}{6}
       \ncput*{\tiny $5$}
     \ncline{2}{3}
       \ncput*{\tiny $1$}
     \ncline{3}{5}
       \ncput*{\tiny $4$}
     \ncline{3}{6}
       \ncput*{\tiny $3$}
     \ncline{4}{5}
       \ncput*{\tiny $2$}
     \ncline{7}{6}
       \ncput*{\tiny $8$}
     \ncline{8}{6}
       \ncput*{\tiny $6$}
     \ncline{8}{9}
       \ncput*{\tiny $7$}
       \psline[linecolor=blue,linearc=.15]{o->}(3,-.9)(1.5,.1)(-.1,-1)(-.1,-2)
       \psline[linecolor=purple,linearc=.15]{o->}(.1,-2)(.1,-1)
       \psline[linecolor=red,linearc=.03]{o->}(.1,-.99)(1.3,-.2)(.9,-1)
       \psline[linecolor=magenta,linearc=.03]{o->}(1.1,-1)(1.49,-.15)
       \psline[linecolor=deeppink,linearc=.15]{o->}(1.51,-.19)(1.9,-1)(1.4,-2)
       \psline[linecolor=green,linearc=.15]{o->}(1.4,-1.99)(1.4,-3)
       \psline[linecolor=darkgreen,linearc=.15]{o->}(1.6,-3)(1.6,-2)(2,-1.2)
       \psline[linecolor=darkred,linearc=.15]{o->}(2.03,-1.2)(2.4,-2)
       \psline[linecolor=orange,linearc=.03]{o->}(2.6,-2)(1.65,-.2)(2.9,-1)
  \end{pspicture}
\end{center}
\caption{A Tree From \Fdown{9} And Its \migt{s}}
\label{GY_tree}
\end{figure}

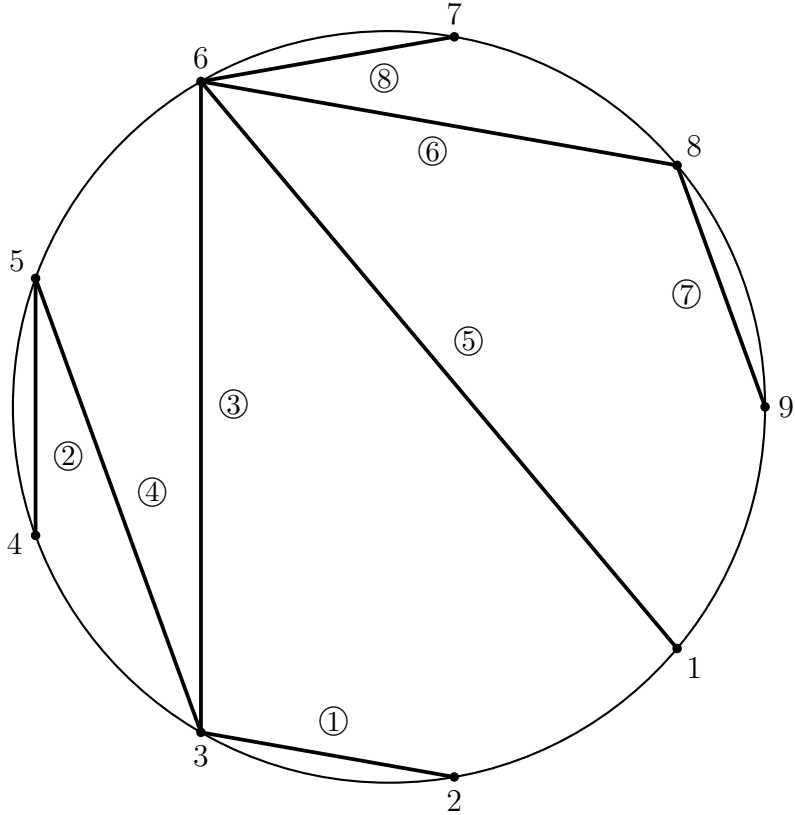
\begin{figure}
\begin{center}
{\psset{unit=5}
\begin{pspicture}(-1.02,-1.05)(1.02,1.05)
  \pscircle(0,0){1}
  \rput(0.766044443118978, -0.642787609686539){\rnode{1}{\psdot(0,0)}}	   
  \uput[315](0.766044443118978, -0.642787609686539){$1$}
  \rput(0.173648177666930, -0.984807753012208){\rnode{2}{\psdot(0,0)}}	   
  \uput[-90](0.173648177666930, -0.984807753012208){$2$}
  \rput(-0.500000000000000, -0.866025403784439){\rnode{3}{\psdot(0,0)}}   
  \uput[270](-0.500000000000000, -0.866025403784439){$3$}
  \rput(-0.939692620785908, -0.342020143325669){\rnode{4}{\psdot(0,0)}}   
  \uput[200](-0.939692620785908, -0.342020143325669){$4$}
  \rput(-0.939692620785908, 0.342020143325669){\rnode{5}{\psdot(0,0)}}	   
  \uput[140](-0.939692620785908, 0.342020143325669){$5$}
  \rput(-0.500000000000000, 0.866025403784439){\rnode{6}{\psdot(0,0)}}	   
  \uput[90](-0.500000000000000, 0.866025403784439){$6$}
  \rput(0.173648177666930, 0.984807753012208){\rnode{7}{\psdot(0,0)}}	   
  \uput[90](0.173648177666930, 0.984807753012208){$7$}
  \rput(0.766044443118978, 0.642787609686539){\rnode{8}{\psdot(0,0)}}	   
  \uput[45](0.766044443118978, 0.642787609686539){$8$}
  \rput(1.00000000000000, 0.000000000000000){\rnode{9}{\psdot(0,0)}}      
  \uput[0](1.00000000000000, 0.000000000000000){$9$}
  \ncline[linewidth=.01]{1}{6}
  \nbput{\raisebox{.5pt}{\textcircled{\raisebox{-.9pt} {\small $5$}}}}
  \ncline[linewidth=.01]{6}{8}
  \nbput{\raisebox{.5pt}{\textcircled{\raisebox{-.9pt} {\small $6$}}}}
  \ncline[linewidth=.01]{6}{7}
  \nbput[npos=.7]{\raisebox{.5pt}{\textcircled{\raisebox{-.9pt} {\small $8$}}}}
  \ncline[linewidth=.01]{6}{3}
  \naput{\raisebox{.5pt}{\textcircled{\raisebox{-.9pt} {\small $3$}}}}
  \ncline[linewidth=.01]{3}{5}
  \nbput{\raisebox{.5pt}{\textcircled{\raisebox{-.9pt} {\small $4$}}}}
  \ncline[linewidth=.01]{4}{5}
  \nbput[npos=.3]{\raisebox{.5pt}{\textcircled{\raisebox{-.9pt} {\small $2$}}}}
  \ncline[linewidth=.01]{2}{3}
  \nbput{\raisebox{.5pt}{\textcircled{\raisebox{-.9pt} {\small $1$}}}}
  \ncline[linewidth=.01]{8}{9}
  \nbput{\raisebox{.5pt}{\textcircled{\raisebox{-.9pt} {\small $7$}}}}
\end{pspicture}}
\end{center}
\caption{Circle Chord Diagram Of The Tree In Figure~\ref{GY_tree}.}
\label{circleChord}
\end{figure}

In \cite{GouldenYong}, multiplication in $\sym_n$ is from right-to-left, however their numbering of 
the transpositions in a transposition sequence is from left-to-right. We wanted both the labeling and the multiplication to go in the same order.  To make our work fit most smoothly with
their work, notice that we have opted to keep their numbering from left-to-right, but
have changed multiplication to also go from left-to-right.
Thus in \cite{GouldenYong}, when they refer to \emph{factorizations of $(1,2, \ldots, n)$
into $n-1$ transpositions}, in our terminology, they are referring to exactly the set 
\Fdown{n} from Definition~\ref{def_FT}.  Recall that from Theorem~\ref{thmDenes} we know 
that the transpositions sequences in \Fdown{n} are trees.  Thus it makes sense to find the 
Circle Chord Diagram of a transposition sequence from \Fdown{n}. 
In \cite{GouldenYong} (see Theorem 2.2) the following theorem is proved.

\begin{thm} \cite{GouldenYong} \label{thm_gy}
For any transposition sequence $\ps{s} \in \Fdown{n}$
its Circle Chord Diagram has the following properties:

\begin{enumerate}


\item The chords are non-crossing.

\item \label{item_decr_clock}
At each of the $n$ vertices on the circle, the labels of the incident chords decrease as we turn clockwise.


\end{enumerate}

\end{thm}  
The properties of the theorem can all be verified of the example in Figure~\ref{circleChord}.  We now give a definition that basically comes from \cite{GouldenYong}, calling it the \emph{Goulden-Yong Dual}; the coherence of the definition depends on Theorem~\ref{thm_gy}.
\begin{defn} \cite{GouldenYong}
Given a tree from \Fdown{n}, its \dword{Goulden-Yong Dual} is determined as follows:

\begin{itemize}

\item Draw its Circle Chord Diagram, which divides the disk into $n$ regions.

\item Place a new vertex in each region, labeling the vertex $k$ if it is in the region
that contains arc $k$.

\item Create an edge between two new vertices if their regions have a chord in common, labeling the edge by the label on the chord.

\end{itemize}

\end{defn}

For example, the Goulden-Yong Dual of the tree in Figure~\ref{GY_tree} is pictured in
Figure~\ref{planarDual}: The dashed edges and smaller vertices depict the original
tree from the Circle Chord Diagram of Figure~\ref{circleChord}, and the solid lines with larger vertices depict its
Goulden-Yong Dual.
Note that the Trail Dual of the tree in Figure~\ref{GY_tree} is exactly
the Goulden-Yong Dual pictured in Figure~\ref{planarDual}; 
we will see that this is generally true in Theorem~\ref{thm_GY_dual}.
As an example of the next lemma, note that the chords of region $2$ of Figure~\ref{circleChord} are the ones labeled $1$, $3$, and $5$, exactly the same as the edges traversed by trail $\wk_2$.

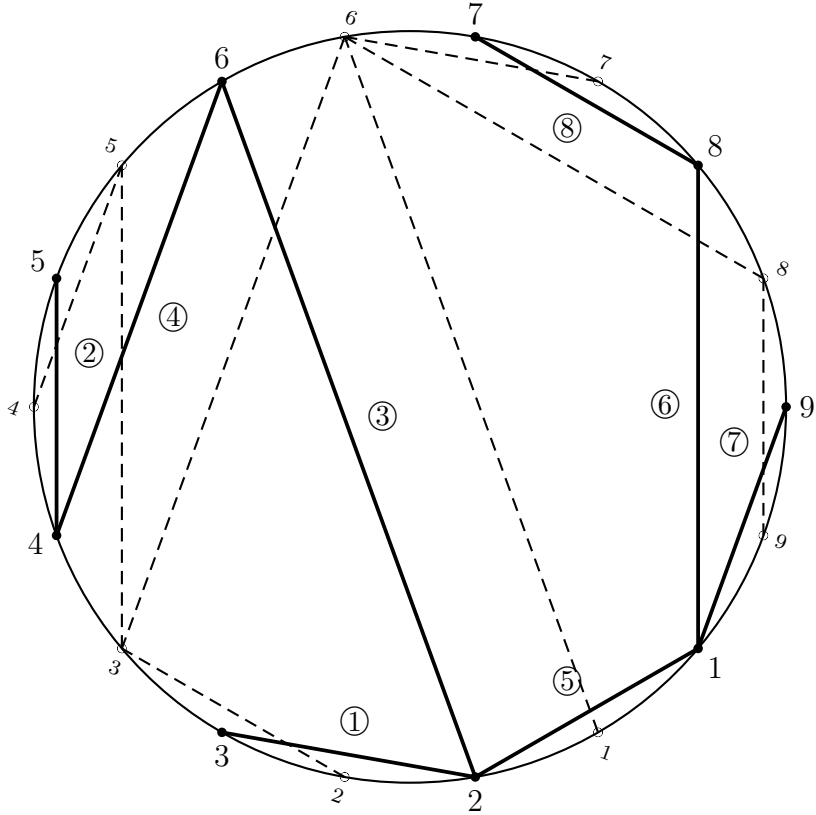
\begin{figure}
\begin{center}
{\psset{unit=5}
\begin{pspicture}(-1.02,-1.05)(1.02,1.05)
{\psset{linestyle=dashed,dotstyle=o}
\psrotate(0,0){-20}{
  \rput(0.766044443118978, -0.642787609686539){\rnode{1}{\psdot(0,0)}}	   
  \uput[315](0.766044443118978, -0.642787609686539){\tiny $1$}
  \rput(0.173648177666930, -0.984807753012208){\rnode{2}{\psdot(0,0)}}	   
  \uput[-90](0.173648177666930, -0.984807753012208){\tiny $ 2$}
  \rput(-0.500000000000000, -0.866025403784439){\rnode{3}{\psdot(0,0)}}   
  \uput[270](-0.500000000000000, -0.866025403784439){\tiny $ 3$}
  \rput(-0.939692620785908, -0.342020143325669){\rnode{4}{\psdot(0,0)}}   
  \uput[200](-0.939692620785908, -0.342020143325669){\tiny $ 4$}
  \rput(-0.939692620785908, 0.342020143325669){\rnode{5}{\psdot(0,0)}}	   
  \uput[140](-0.939692620785908, 0.342020143325669){\tiny $ 5$}
  \rput(-0.500000000000000, 0.866025403784439){\rnode{6}{\psdot(0,0)}}	   
  \uput[90](-0.500000000000000, 0.866025403784439){\tiny $ 6$}
  \rput(0.173648177666930, 0.984807753012208){\rnode{7}{\psdot(0,0)}}	   
  \uput[90](0.173648177666930, 0.984807753012208){\tiny $ 7$}
  \rput(0.766044443118978, 0.642787609686539){\rnode{8}{\psdot(0,0)}}	   
  \uput[45](0.766044443118978, 0.642787609686539){\tiny $ 8$}
  \rput(1.00000000000000, 0.000000000000000){\rnode{9}{\psdot(0,0)}}      
  \uput[0](1.00000000000000, 0.000000000000000){\tiny $ 9$}
  \ncline{1}{6}
  \ncline{6}{8}
  \ncline{6}{7}
  \ncline{6}{3}
  \ncline{3}{5}
  \ncline{4}{5}
  \ncline{2}{3}
  \ncline{8}{9}
}}
 \pscircle[linestyle=solid](0,0){1}
  \rput(0.766044443118978, -0.642787609686539){\rnode{1}{\psdot(0,0)}}	   
  \uput[315](0.766044443118978, -0.642787609686539){$1$}
  \rput(0.173648177666930, -0.984807753012208){\rnode{2}{\psdot(0,0)}}	   
  \uput[-90](0.173648177666930, -0.984807753012208){$2$}
  \rput(-0.500000000000000, -0.866025403784439){\rnode{3}{\psdot(0,0)}}   
  \uput[270](-0.500000000000000, -0.866025403784439){$3$}
  \rput(-0.939692620785908, -0.342020143325669){\rnode{4}{\psdot(0,0)}}   
  \uput[200](-0.939692620785908, -0.342020143325669){$4$}
  \rput(-0.939692620785908, 0.342020143325669){\rnode{5}{\psdot(0,0)}}	   
  \uput[140](-0.939692620785908, 0.342020143325669){$5$}
  \rput(-0.500000000000000, 0.866025403784439){\rnode{6}{\psdot(0,0)}}	   
  \uput[90](-0.500000000000000, 0.866025403784439){$6$}
  \rput(0.173648177666930, 0.984807753012208){\rnode{7}{\psdot(0,0)}}	   
  \uput[90](0.173648177666930, 0.984807753012208){$7$}
  \rput(0.766044443118978, 0.642787609686539){\rnode{8}{\psdot(0,0)}}	   
  \uput[45](0.766044443118978, 0.642787609686539){$8$}
  \rput(1.00000000000000, 0.000000000000000){\rnode{9}{\psdot(0,0)}}      
  \uput[0](1.00000000000000, 0.000000000000000){$9$}
  \ncline[linewidth=.01]{1}{9}
  \naput[npos=.8]{\raisebox{.5pt}{\textcircled{\raisebox{-.9pt} {\small $7$}}}}
  \ncline[linewidth=.01]{1}{8}
  \naput{\raisebox{.5pt}{\textcircled{\raisebox{-.9pt} {\small $6$}}}}
  \ncline[linewidth=.01]{1}{2}
  \nbput{\raisebox{.5pt}{\textcircled{\raisebox{-.9pt} {\small $5$}}}}
  \ncline[linewidth=.01]{2}{6}
  \nbput{\raisebox{.5pt}{\textcircled{\raisebox{-.9pt} {\small $3$}}}}
  \ncline[linewidth=.01]{2}{3}
  \nbput{\raisebox{.5pt}{\textcircled{\raisebox{-.9pt} {\small $1$}}}}
  \ncline[linewidth=.01]{6}{4}
  \naput{\raisebox{.5pt}{\textcircled{\raisebox{-.9pt} {\small $4$}}}}
  \ncline[linewidth=.01]{4}{5}
  \nbput[npos=.7]{\raisebox{.5pt}{\textcircled{\raisebox{-.9pt} {\small $2$}}}}
  \ncline[linewidth=.01]{7}{8}
  \nbput{\raisebox{.5pt}{\textcircled{\raisebox{-.9pt} {\small $8$}}}}
\end{pspicture}}
\end{center}
\caption{Goulden-Yong Dual (Solid Lines) Of The Tree In Figures~\ref{GY_tree} And
\ref{circleChord} (Dashed Lines).}
\label{planarDual}
\end{figure}

\begin{lem} \label{lem_cir_walk}
Suppose $\graph{T} \in \Fdown{n}$ and $C$ is its
Circle Chord Diagram.  Suppose $x \in [n]$.
Then the edges in $\wk_x$ are exactly the edges on the boundary in region $x$ of $C$.
\end{lem}

\begin{proof} 

Suppose the trajectory of $\wk_x$ is \seq{x, x_1, \ldots, x_k}.
Let $e$ be the most clockwise edge at $x$ (for example, in Figure~\ref{circleChord}, edge $3$ is the most clockwise edge at vertex $6$).
By property~\ref{item_decr_clock} of Theorem~\ref{thm_gy}, $\wk_x$ moves along edge
$e$ from $x$ to $x_1$.  Then, again by property~\ref{item_decr_clock}, the trail goes from $x_1$ to $x_2$, along the edge that is one chord counter-clockwise from $e$, when turning at $x_1$ (for example, in Figure~\ref{circleChord}, at vertex $3$, the chord labeled by $4$ is one chord counter-clockwise from the chord labeled $3$).   As we continue we see that $\wk_x$
traverses one of the $n$ regions of $C$, moving along its boundary in a clockwise fashion, starting at $x$ on the circle and ending at $x-1$ (understanding vertex $0$ to be the same as vertex $n$).  That is, $\wk_x$ consists of exactly the edges of region $x$.
\end{proof}

\begin{thm} \label{thm_GY_dual}
For any tree from \Fdown{n}, its Goulden-Yong Dual is the same as its dual.

\end{thm}

\begin{proof}

Consider some tree $\graph{T} \in \Fdown{n}$, and let $C$ be its Circle Chord Diagram.
Let $\graph{T'}$ be its Trail Dual and $\graph{T^*}$ its Goulden-Yong Dual.
Both $\graph{T'}$ and $\graph{T^*}$ are labeled graphs with vertex set $[n]$, so it suffices to observe that for any distinct
$u, v \in [n]$, we have the following equivalences
(where the second one follows by Lemma~\ref{lem_cir_walk}).

\vspace{3mm}

\noindent
$\{u,v\} \hbox{ is an edge in } \graph{T'} \hbox{ with label } k$

\hspace{10mm} $\Leftrightarrow  \wk_u \hbox{ and } \wk_v \hbox{ both use edge } k$

\hspace{10mm} $\Leftrightarrow \hbox{Chord } k \hbox{ is on the boundary of region } u \hbox{ and region } v.$

\hspace{10mm} $\Leftrightarrow \hbox{In } \graph{T}^* \hbox{ there is an edge labeled } k \hbox{ between } u \hbox{ and } v.$  
\end{proof}

\section{Conclusion and Future Work}

In this paper  we focused on minimal transitive factorizations of the permutation $(n, \ldots, 2, 1)$, investigating an interesting bijection.  In \cite{GouldenJackson1997} a general formula is found for the number of minimal transitive factorizations of any permutation.  Based on this result, they motivate the search for interesting bijections between such sets of factorizations and other sets of combinatorial interest.  Making progress on this program, \cite{KimSeo2003}
found a bijection for the minimal transitive factorizations of $(1)(2, \ldots, n)$, and 
\cite{Rattan2006} found bijections for $(1, 2)(3, \ldots, n)$ and $(1, 2, 3)(4, \ldots, n)$; both papers used parking functions.
Our hope is that our alternative definitions of the dual, which apply to any graph (not just trees), could be a useful tool for such research.

\section{Acknowledgments}
K. Ojakian was supported by a PSC-CUNY Research Award (Traditional A).

\bibliography{BibOfFactorizationsAndGraphs}
\bibliographystyle{plain}

\end{document}